%% file: ms.tex
\documentclass[11pt,a4paper]{article}
\pdfoutput=1 
\input{./packages.tex}
\input{./symbols.tex}

\title{Almost Commutative Manifolds and Their Modular Classes}
\author{Shuichi Harako}
\date{}
\begin{document}
\maketitle
\input{abstract.tex}

\input{section_introduction.tex}

\input{section_almost_commutative_algebras.tex}

\input{section_matrix.tex}

\input{section_function.tex}
\input{section_almost_commutative_manifolds.tex}

\input{section_volume.tex}

\input{section_modular_class.tex}
%
\input{references.tex}
%
\par \ 
\par \textsc{Graduate School of Mathematical Sciences, the University of Tokyo, 3-8-1 Komaba, Meguro-ku, Tokyo, 153-8914, Japan}
\par \textit{E-mail address}: \texttt{harako@ms.u-tokyo.ac.jp}
\end{document}

%% file: packages.tex
\usepackage{autobreak,amsmath,amssymb,array,ascmac,bbm,bigdelim,graphicx,mathrsfs,mathtools,multicol,multirow,stmaryrd}
\usepackage[normalem]{ulem}
\usepackage[inline]{enumitem}
\usepackage[x11names]{xcolor}

\usepackage{tikz}
\usetikzlibrary{cd,calc,arrows.meta}
\usetikzlibrary{decorations.markings,decorations.pathreplacing}
\tikzset{
    longbrace/.style = {
        decorate, decoration = {brace, amplitude=10pt}
    },
}

\usepackage[setpagesize=false,linktocpage]{hyperref}

\numberwithin{equation}{section}
\usepackage{docmute}

\usepackage[amsmath,hyperref,thref,thmmarks]{ntheorem}
\theoremstyle{plain}
\theoremseparator{.}
\theorembodyfont{\normalfont}
\newtheorem{definition}{Definition}[section]
\newtheorem{proposition}[definition]{Proposition}
\newtheorem{lemma}[definition]{Lemma}

\newtheorem{remark}[definition]{Remark}
\renewtheorem*{remark*}{Remark}
\newtheorem{example}[definition]{Example}
\renewtheorem*{example*}{Example}
\newtheorem*{notations}{Notations}
\newtheorem*{acknowledgments}{Acknowledgments}

\theoremstyle{plain}
\theoremheaderfont{\normalfont}
\theoremsymbol{\ensuremath{\blacksquare}}
\newtheorem*{proof}{\textit{Proof}}

%% file: symbols.tex
\newcommand{\defeq}{\coloneqq}
\newcommand{\eqdef}{\eqqcolon}

\newcommand{\isom}{\cong}

\newcommand{\tensor}{\otimes}

\DeclareMathOperator{\Kernel}{Ker}
\DeclareMathOperator{\Image}{Im}

\DeclareMathOperator{\Der}{Der}

\DeclareMathOperator{\Aut}{Aut}
\DeclareMathOperator{\Ber}{Ber}

\newcommand{\Real}{\mathbb{R}}
\newcommand{\Complex}{\mathbb{C}}
\newcommand{\ZInteger}{\mathbb{Z}}

\newcommand{\Korper}{\mathbb{K}}

\newcommand{\gee}{\mathfrak{g}}

\newcommand{\abs}[1]{ { \left| #1 \right| } }
\newcommand{\setmid}{ \ \middle| \ }
\newcommand{\setin}[2]{{\left\{ {#1} \setmid {#2} \right\}}}
\newcommand{\restr}[2]{{\left. {#1} \right|_{#2}}}

\DeclareMathOperator{\vol}{vol}

\DeclareMathOperator{\Vect}{Vect}

\DeclareMathOperator{\GL}{GL}
\DeclareMathOperator{\trace}{tr}

\DeclareMathOperator{\rhodet}{\rho\det}
\DeclareMathOperator{\rhotr}{\rho\trace}
\DeclareMathOperator{\rhoBer}{\rho\Ber}
\DeclareMathOperator{\Diverg}{Div}
\DeclareMathOperator{\Modular}{Mod}

\newcommand{\func}{\mathscr{O}}
\newcommand{\rhofunc}[1][\rho]{\mathscr{O}_{#1}}
\newcommand{\dbrack}[1]{{\llbracket {#1} \rrbracket}}
\newcommand{\even}{\mathrm{even}}
\newcommand{\odd}{\mathrm{odd}}

\newcommand{\ppx}[2][]{\frac{\partial {#1}}{\partial {#2}}}

\newcommand{\transpose}[1]{ {{#1}^{\top}} }

\newcommand{\defnamelessmap}[5][]{%
    {#1} %
    \begin{tikzcd}[ampersand replacement=\&,row sep=0pt,baseline=(S.base)]%
        |[alias=S]|{#2} \arrow[r, rightarrow] \& {#3} \\%
        {#4} \arrow[r, mapsto] \& {#5}%
    \end{tikzcd}%
}

%% file: abstract.tex
\begin{abstract}
    An almost commutative algebra, or a \(\rho\)-commutative algebra, 
        is an algebra which is graded by an abelian group 
        and whose commutativity is controlled by a function 
        called a commutation factor. 
    The same way as a formulation of a supermanifold as a ringed space, 
        we introduce concepts of the \(\rho\)-commutative versions of 
        manifolds, Q-manifolds, Berezin volume forms, and the modular classes. 
    They are generalizations of the ones in supergeometry. 
    We give examples including 
        a \(\rho\)-commutative version of the Schouten bracket 
        and a noncommutative torus. 
\end{abstract}

%% file: section_introduction.tex
\section{Introduction}
For the study of noncommutative algebras, 
it is convenient to impose a constraint for the noncommutativity of the algebras, 
which is often called a commutation rule. 
A commutative algebra is an algebra with a commutation rule \(fg=gf\) 
for all element \(f, g\) of the algebra. 
A superalgebra is a \(\ZInteger / 2\ZInteger\)-graded algebra which has a commutation rule 
\(fg=-gf\) if homogeneous elements \(f, g\) are odd, and \(fg=gf\) otherwise.
A \(\rho\)-commutative (or almost commutative, \(\epsilon\)-commutative) algebra 
is an algebra graded by an arbitrary abelian group \(G\) 
with its commutation rule controlled by 
a map called a commutation factor \cite{Scheunert1979,Bongaarts1994,Ciupala2005,Bruce2020}. 
This condition characterizing noncommutativity covers not only commutative algebras or superalgebras, 
but quaternions, quantum planes, and noncommutative tori, etc. 
\(\rho\)-Lie algebras are also defined as the \(\rho\)-commutative version of Lie algebras. 
\par Meanwhile, we have a characteristic class called the modular class of a Q-manifold. 
A Q-manifold or a dg-manifold is a supermanifold with an odd vector field squared to zero 
    as a derivation on functions. 
This condition for a vector field appears in various cases. 
Lie algebroids, \(L_\infty\)-algebras, and the de Rham or Dolbeault complexes 
    are formulated as Q-manifolds \cite{Kontsevich1999}. 
In the aspect of mathematical physics, the classical BRST formalism 
    is one of the application of Q-manifolds \cite{Mnev2019}. 
Theories of characteristic classes of Q-manifolds are studied by Kotov \cite{Kotov2007}, 
Lyakhovich-Mosman-Sharapov \cite{Lyakhovich2009}, and Bruce \cite{Bruce2017}. 
The modular class is one of them which generalizes the modular class of 
    a higher Poisson manifold. 
\par The main goal of this paper is to introduce a concept of 
the modular classes for \(\rho\)-Q-manifolds 
and give examples of them. 
Supermanifolds are constructed by replacing the local functional algebras of the underlying manifold 
    with a graded-commutative \(\ZInteger / 2\ZInteger\)-algebra. 
We apply the same procedure to a manifold but with a \(\rho\)-commutative algebra. 
We call the resulting manifold a \(\rho\)-manifold. 
Vector fields on a \(\rho\)-manifold are \(\rho\)-derivations on 
    the \(\rho\)-commutative functional algebra. 
An odd vector field squared to zero in this case defines a \(\rho\)-Q-manifold. 
The \(\rho\)-commutative version of a Berezin volume form is defined, which 
    allows us to define the modular class of a \(\rho\)-Q-manifold 
    the same way as a Q-manifold. 
We give examples of \(\rho\)-Q-manifolds and their modular classes, e.g., 
    the tangent or cotangent bundle of a \(\rho\)-manifold, 
    the degree shift of them, and the BRST differential on noncommutative tori. 
\par In Section 2, we review the theory of \(\rho\)-commutative algebras 
    and \(\rho\)-Lie algebras. We mainly give definitions and examples of them here. 
In Section 3, we review the matrix algebra whose entries are in a \(\rho\)-commutative algebra. 
We introduce the determinant and the Berezinian in this version 
    by \cite{KobayashiNagamachi1984,Covolo2016}. 
In Section 4, we discuss the specific algebra which is a subalgebra of 
    the set of formal power series with \(G\)-graded indeterminates. 
    This section is devoted to the preparation for argument of \(\rho\)-manifolds. 
In Section 5, we define a \(\rho\)-manifold and a \(\rho\)-Q-manifold. 
Here we discuss the de Rham complex and the degree-\(i\) Schouten bracket 
    of a \(\rho\)-manifold. 
In Section 6, we introduce the Berezinian bundle of a \(\rho\)-manifold and its orientability 
    using the result of Section 3. 
In section 7, we define the modular class of a \(\rho\)-Q-manifold 
    and calculate them in several situations, for example, 
    the degree-\(i\) Schouten bracket and noncommutative tori.  

\begin{notations}
    We write \(\Korper = \Real\) or \(\Complex\) for a coefficient field. 
    The underlying manifold \(M\) can be a smooth or real analytic or holomorphic manifold. 
    The term `manifold' in this paper means either of them unless otherwise specified. 
    \(\func(M)\) means the algebra of real-valued smooth functions \(C^\infty (M, \Real)\) 
        or complex-valued smooth functions \(C^\infty (M, \Complex)\) if \(M\) is smooth, 
        real-analytic functions \(C^\omega (M, \Real)\) if \(M\) is real-analytic, 
        holomorphic functions \(C^\omega (M, \Complex)\) if \(M\) is holomorphic. 
    For \(\Korper = \Real\) and \(\func(M)\) being the set of 
        a real-valued smooth or real-analytic functions, 
        we call this case the \textit{real category}. 
    For \(\Korper = \Complex\) and \(\func(M)\) being the set of 
        a complex-valued smooth or holomorphic functions, 
        we call this case the \textit{complex category}. 
\end{notations}

\begin{acknowledgments}
    We are very grateful to our supervisor Takuya Sakasai and our colleagues in laboratory 
        for useful advice and comments. 
    \par This work was supported by JSPS KAKENHI Grant Number 22J13678. 
\end{acknowledgments}

%% file: section_almost_commutative_algebras.tex
\section{\(\rho\)-commutative algebras and \(\rho\)-Lie algebras}
In this section, we briefly review the theory of \(\rho\)-commutative algebras 
and \(\rho\)-Lie algebras.
\begin{definition}
    Let \(G\) be an abelian group. 
    A \textit{commutation factor on \(G\)} is a map \(\rho \colon G \times G \to \Korper\) 
    satisfying the following conditions: 
    \begin{itemize}
        \item[(i)] \(\rho(i, j) \rho(j, i) = 1\) for all \(i, j \in G\), 
        \item[(ii)] \(\rho(i+j, k) =  \rho(i, k) \rho (j, k)\) for all \(i,j,k \in G\).  
    \end{itemize}
\end{definition} 
Directly from this definition, it is easily shown that 
\begin{gather*}
    \rho(i, j) \neq 0, \qquad \rho(i, j) = \rho(j, i)^{-1} = \rho (-j, i) = \rho (j, -i), \qquad 
    \rho(i, i) = \pm 1, \\
    \rho(0, i) = \rho(i, 0) = 1, \qquad \rho(i, j+k) =  \rho(i, j) \rho (i, k)
\end{gather*}
for all \(i, j, k \in G\). 
\par Let \(G\) be an abelian group and 
\(A = \bigoplus_{i \in G} A_i\) be a \(G\)-graded algebra over \(\Korper\). 
We write \(\abs{f} \defeq i\) if and only if \(f \in A_i\). 
Unless otherwise mentioned, we assume \(f\) is a \(G\)-homogeneous element when we consider the degree \(\abs{f}\).
\begin{definition}
    A \(G\)-graded algebra \(A\) is said to be a \textit{\(\rho\)-commutative algebra} if 
    \begin{equation*}
        fg = \rho(\abs{f}, \abs{g})gf
    \end{equation*}
    holds for all \(f,g \in A\). 
\end{definition}

\begin{example}
    For any abelian group \(G\) and for the trivial commutation factor 
    \(\rho \colon G \times G \to \{1\} \hookrightarrow \Korper\), 
    a \(\rho\)-commutative algebra is simply an usual commutative algebra 
        with an extra \(G\)-grading. 
\end{example}

\begin{example}\label{example_supercommutative_algebra}
    Let \(G = \ZInteger / 2\ZInteger\) and define  
    \(\rho \colon \ZInteger / 2\ZInteger \times \ZInteger / 2\ZInteger \to \Korper\) by 
    \begin{equation*}
        \rho(i, j) \defeq (-1)^{ij} 
    \end{equation*} 
    for all \(i, j \in G\).
    In this case, a \(\rho\)-commutative algebra coincides with a superalgebra. 
\end{example}

\begin{example}\label{example_noncommutative_torus_algebra}
    Fix an integer \(m > 0\) and a skew-symmetric matrix \(\Theta = (\theta_{kl})_{k,l} \in M(m, \Real)\). 
    Let \(\Korper = \Complex\) and \(G = \ZInteger^m\). 
    Write an element of \(G\) like \(\mathbf{i} = (i_1, i_2, \dots, i_m) \in G\). 
    Define a map \(\rho \colon G \times G \to \Complex\) by 
    \begin{equation*}
        \rho(\mathbf{i}, \mathbf{j}) \defeq \exp(2\pi\sqrt{-1} \ \transpose{\mathbf{i}} \Theta \mathbf{j})
    \end{equation*}
    for all \(\mathbf{i}, \mathbf{j} \in G\). 
    This map becomes a commutation factor on \(\ZInteger^m\) because the skew-symmetricity of \(\Theta\) implies that 
    the equations 
    \begin{align*}
        \rho(\mathbf{i}, \mathbf{j}) \rho(\mathbf{j}, \mathbf{i}) 
        &{}= \exp(2\pi\sqrt{-1} \ (\transpose{\mathbf{i}} \Theta \mathbf{j} + \transpose{\mathbf{j}} \Theta \mathbf{i})) \\
        &{}= \exp(2\pi\sqrt{-1} \ (\transpose{\mathbf{i}} \Theta \mathbf{j} - \transpose{\mathbf{i}} \Theta \mathbf{j}))
        = 1
    \end{align*}
    and 
    \begin{align*}
        \rho(\mathbf{i}+\mathbf{j}, \mathbf{k}) 
        &{}= \exp(2\pi\sqrt{-1} \ (\transpose{(\mathbf{i} + \mathbf{j})} \Theta \mathbf{k})) \\
        &{}= \exp(2\pi\sqrt{-1} \ (\transpose{\mathbf{i}} \Theta \mathbf{k} + \transpose{\mathbf{j}} \Theta \mathbf{k}))
        = \rho(\mathbf{i}, \mathbf{k}) \rho(\mathbf{j}, \mathbf{k})
    \end{align*}
    hold for all \(\mathbf{i}, \mathbf{j}, \mathbf{k} \in G\). 
    \par Let \(A_\Theta\) be an algebra generated by indeterminates \(u^1, u^2, \dots, u^m\) with relations
    \begin{equation*}
        u^k u^l = \exp(2\pi\sqrt{-1} \ \theta_{kl}) u^l u^k
    \end{equation*}
    for \(1 \leq k, l \leq m\). 
    This algebra has a natural \(G\)-grading 
    \(A_\Theta = \bigoplus_{\mathbf{i} \in G} (A_\Theta)_{\mathbf{i}}\)
    where 
    \begin{equation*}
        (A_\Theta)_{\mathbf{i}} \defeq 
        \begin{cases}
            \Complex \cdot \{(u^1)^{i_1} (u^2)^{i_2} \cdots (u^m)^{i_m}\} 
                & \text{if } i_k \geq 0 \text{ for all } 1 \leq k \leq m, \\ 
            0 & \text{otherwise}.
        \end{cases}
    \end{equation*}
    Then, \(A_\Theta\) becomes a \(\rho\)-commutative algebra over \(\Complex\) 
    with respect to the commutation factor \(\rho\) defined above. 
    Moreover, \(A_\Theta\) is nothing but the algebra of functions on 
    the \(m\)-dimensional noncommutative torus \cite{Rieffel1990}. 
    The case \(m=2\) is discussed in \cite{Bruce2020}. 
\end{example}

\begin{remark}
    When a \(G\)-graded \(\rho\)-commutative algebra is given, 
    we often consider another algebra associated with it 
    which is graded by \(G^\prime \defeq \ZInteger \times G\), 
    and whose commutation factor \(\rho^\prime \colon G^\prime \times G^\prime \to \Korper\) is defined by 
    \begin{equation*}
        \rho^\prime ((s, i), (t, j)) \defeq (-1)^{st} \rho(i, j)
    \end{equation*}
    for \((s, i), (t, j) \in G^\prime\). 
    In this case, the \(G^\prime\)-degree on a \(\rho^\prime\)-algebra is written by 
        \(\abs{\cdot}^\prime\). 
    We use these notations from now on. 
\end{remark} 

\begin{example}
    Suppose \(V = \bigoplus_{i \in G} V_i\) is a \(G\)-graded vector space 
    and \(\rho\) is a commutation factor on \(G\). 
    The \textit{\(\rho\)-symmetric algebra} 
    \(S_\rho^\bullet V \defeq \bigoplus_{k \geq 0} S_\rho^k V\) \textit{on \(V\)} is defined by 
    \begin{equation*}
        S_\rho^k V \defeq V^{\tensor k} / I 
    \end{equation*}
    for \(k \geq 0\) where \(I\) is the two-sided ideal of \(V^{\tensor k}\) 
    generated by 
    \begin{equation*}
        \setin{v \tensor w-\rho(\abs{v},\abs{w})w \tensor v}{v,w \in V}.
    \end{equation*}
    This algebra is clearly a \(\rho\)-commutative algebra. 
    \par On the other hand, the \textit{\(\rho\)-antisymmetric algebra} 
    \(\wedge_\rho^\bullet V \defeq \bigoplus_{k \geq 0} \wedge_\rho^k V\) \textit{on \(V\)}, 
    which is defined by 
    \begin{equation*}
        \wedge_\rho^k V \defeq V^{\tensor k} / J 
    \end{equation*}
    for \(k \geq 0\) where \(J\) is the ideal 
    generated by 
    \begin{equation*}
        \setin{v \tensor w + \rho(\abs{v},\abs{w})w \tensor v}{v,w \in V},
    \end{equation*}
    is not \(\rho\)-commutative. 
    However, it is regarded as a \(G^\prime\)-graded \(\rho^\prime\)-commutative algebra 
    if we consider a \(G^\prime\)-grading 
    \(\abs{\omega}^\prime \defeq (k, \abs{\omega})\) for \(\omega \in \wedge_\rho^k V\). 
\end{example}
\begin{example}
    Let \(R\) be a usual commutative algebra over \(\Korper\) 
        and \(x^1, \allowbreak \dots, \allowbreak x^m\) be \(G\)-graded indeterminates. 
    A well-known construction of formal power series 
    \begin{equation*}
        R \dbrack{x^1, \dots, x^m}
        \defeq \setin{\sum_{\mathbf{w} \in \ZInteger_{\geq 0}^{\times m}}
            f_{\mathbf{w}}x^{\mathbf{w}}}{
                f_{\mathbf{w}} \in R \text{ for every } \mathbf{w}}, 
    \end{equation*}
    where \(x^{\mathbf{w}}\) is a multi-index notation 
    \(x^{\mathbf{w}} \defeq (x^1)^{w_1} (x^2)^{w_2} \cdots (x^m)^{w_m}\) 
    for \(\mathbf{w} = (w_1, w_2, \dots, w_m) \in \ZInteger_{\geq 0}^{\times m}\), 
    does not provide a \(\rho\)-commutative algebra in general. 
    It is a direct product of \(G\)-homogeneous subspaces 
    \begin{equation*}
        R \dbrack{x^1, \dots, x^m} = \prod_{i \in G} R \dbrack{x^1, \dots, x^m}_i, 
    \end{equation*}
    not a direct sum if \(G\) is an infinite group. Here, for every \(i \in G\), 
    \begin{align*}
        &R \dbrack{x^1, \dots, x^m}_i \\
        \defeq{}& \setin{\sum_{\mathbf{w} \in \ZInteger_{\geq 0}^{\times m}}
            f_{\mathbf{w}}x^{\mathbf{w}} \in R \dbrack{x^1, \dots, x^m}}{
            |x^{\mathbf{w}}| = i \text{ for every } \mathbf{w}}.
    \end{align*}
    For this reason, we consider a direct sum 
    \begin{equation*}
        R \dbrack{x^1, \dots, x^m}_\bullet \defeq \bigoplus_{i \in G} R \dbrack{x^1, \dots, x^m}_i 
    \end{equation*}
    rather than a direct product. 
    This subspace \(R \dbrack{x^1, \dots, x^m}_\bullet\) of \(R \dbrack{x^1, \dots, x^m}\) 
        is certainly a \(\rho\)-commutative algebra. 
\end{example}
\begin{definition}
    Let \(G\) be an abelian group and \(\rho \colon G \times G \to \Korper\) be a commutation factor. 
    A \(G\)-graded vector space \(\gee = \bigoplus_{i \in G} \gee_i\) over \(\Korper\) is said to be 
    a \textit{\(\rho\)-Lie algebra of degree \(d \in G\)} if it is equipped with a \(\Korper\)-bilinear map 
    \([-,-]_\rho \colon \gee \times \gee \to \gee\) satisfying 
    \begin{itemize}
        \item[(i)] \(\abs{\gee} \defeq \abs{[-,-]_\rho} = d\), that is, \(\abs{[f, g]_\rho} = \abs{f} + \abs{g} + d\), 
        \item[(ii)] \([f, g]_\rho = -\rho(\abs{f}, \abs{g}) [g, f]_\rho\) \quad (\(\rho\)-antisymmetricity), 
        \item[(iii)] \([f,[g,h]]_\rho = [[f,g],h]_\rho + \rho(\abs{f}+d,\abs{g}+d)[g,[f,h]]_\rho\)
            \quad (\(\rho\)-Jacobi identity)
    \end{itemize}
    for all \(f, g, h \in \gee\). 
    \par In particular, a \(\rho\)-Lie algebra \(\gee\) of degree \(d\) is said to be 
    a \textit{\(\rho\)-Lie antialgebra} if \(\rho(d, d) = -1\). 
\end{definition}

A \(\rho\)-Lie algebra is originally introduced by Scheunert \cite{Scheunert1979} as an ``\(\epsilon\)Lie algebra'' 
and also known as a Lie colored algebra \cite{Ritcher2001}. 
A \(\rho\)-commutative algebra also has another name as 
an \(\epsilon\)-commutative algebra \cite{Scheunert1979}, 
a colored algebra \cite{Rittenberg1978}, or 
an almost commutative algebra \cite{Bongaarts1994,Bruce2020}, etc. 

\begin{example}
    Any \(G\)-graded algebra \(A\) with a commutation factor \(\rho\) on \(G\)
    becomes a \(\rho\)-Lie algebra of degree 0 by the \textit{\(\rho\)-commutator} on \(A\), 
    which is a \(\Korper\)-bilinear map \([-,-]_\rho \colon A \times A \to A\) defined by 
    \begin{equation*}
        [f, g]_\rho \defeq fg - \rho(\abs{f},\abs{g}) gf
    \end{equation*}
    for \(f, g \in A\).
\end{example}

\begin{example}
    Let \(A = \bigoplus_{i \in G} A_i\) be a \(\rho\)-commutative algebra. 
    A homomorphism \(X \colon A \to A\) of degree \(\abs{X} \in G\) 
    is called a \textit{\(\rho\)-derivation} if 
    \begin{equation*}
        X(fg) = (Xf)g + \rho(\abs{X}, \abs{f})f(Xg)
    \end{equation*}
    holds for all \(f, g \in A\). 
    The vector space consisting of all \(\rho\)-derivations on \(A\) is denoted by \(\rho\Der(A)\), which has 
    a natural \(G\)-grading \(\rho\Der(A) = \bigoplus_{i \in G} \rho\Der(A)_i\) where 
    \begin{equation*}
        \rho\Der(A)_i \defeq \setin{X \in \rho\Der(A)}{ \abs{X} = i}. 
    \end{equation*}
    The space \(\rho\Der(A)\) of \(\rho\)-derivations becomes a \(\rho\)-Lie algebra of degree 0 by 
    the \(\rho\)-commutator 
    \begin{equation*}
        [X, Y]_\rho(f) = X(Yf) - \rho(\abs{X}, \abs{Y})Y(Xf). 
    \end{equation*}
    Moreover, it has a structure of a well-defined left \(A\)-module  
    \begin{equation*}
        (fX)(g) \defeq f(Xg) \qquad (X \in \rho\Der(A), \ f,g \in A)
    \end{equation*}
    by the \(\rho\)-commutativity of \(A\). 
\end{example}

Here we introduce a \(\rho\)-commutative version of a differential graded algebra or a Q-algebra \cite{Bruce2020}. 
\begin{definition}
    A pair \((A, Q)\) is called a \textit{\(\rho\)-commutative Q-algebra} 
        (or \textit{\(\rho\)-Q-algebra} for short) if 
    \begin{itemize}
        \item[(i)] \(A\) is a \(\rho\)-commutative algebra, and \(Q \in \rho\Der(A)\), 
        \item[(ii)] \(\rho(\abs{Q}, \abs{Q})=-1\), 
        \item[(iii)] \([Q,Q]_\rho(=2Q \circ Q)=0\).  
    \end{itemize}
    This derivation \(Q\) is called a \textit{homological \(\rho\)-derivation}. 
\end{definition}
\begin{example}\label{example_trivial_rho_Q_algebra}
    For any \(\rho\)-commutative algebra \(A\), the pair \((A, 0)\) is a \(\rho^\prime\)-Q-algebra. 
    Note that we cannot simply say that \(\rho(|Q|,|Q|) = -1\) 
        because \(G\) might not admit a commutation factor \(\rho\)
        with \(\rho(i, i) = -1\) for some \(i \in G\). 
    For example, any commutation factor on \(\ZInteger / 3\ZInteger\) is trivial. 
\end{example}
\begin{example}
    Let \(\gee\) be an \(n\)-dimensional \(G\)-graded \(\rho\)-Lie algebra over \(\Korper\). 
    Take a basis \((e_1, e_2, \dots, e_m)\) of \(\gee\) and 
    denote the structure constant by \(\{\gamma_{ab}^c\}_{a,b,c}\), i.e. it satisfies 
    \begin{equation*}
        [e_a, e_b]_\rho = \sum_{c=1}^m \gamma_{ab}^c e_c
    \end{equation*}
    for all \(1 \leq a, b \leq m\). 
    Set \(\Pi \gee^\ast \defeq \gee^\ast\) as a set, with a \(G^\prime\)-grading 
    \(\abs{\xi}^\prime \defeq (1, \abs{\xi})\) for \(\xi \in \Pi \gee^\ast\) 
    where the \(G\)-grading on the vector space \(\gee^\ast\) is a natural one induced by 
    the \(G\)-grading \(\abs{\cdot}\) on \(\gee\). 
    \par Then, the algebra \(S_{\rho^\prime}^\bullet (\Pi \gee^\ast)\) has 
    a homological \(\rho^\prime\)-derivation 
    \begin{equation*}
        Q \defeq \frac{1}{2} \sum_{c=1}^m \gamma_{ab}^c \xi^a \xi^b \ppx{\xi^c} 
    \end{equation*}
    where \((\xi^1, \xi^2, \dots, \xi^m)\) is the dual basis of \((e_1, e_2, \dots, e_m)\). 
    In fact, the condition \([Q, Q]_\rho = 0\) is equivalent to 
    the \(\rho\)-Jacobi identity of the original bracket on \(\gee\). 
\end{example}
\par A \(\rho\)-commutative Q-algebra is a generalization of a Q-algebra in superalgebras. 
It is a special case where a \(\rho\)-commutative algebra \(A\) 
is the one in \thref{example_supercommutative_algebra}.

%% file: section_matrix.tex
\section{The graded matrix algebras}
\par Let \(A\) be a \(G\)-graded \(\rho\)-commutative algebra. 
Assume \(G\) is finitely generated. 
\begin{definition}
    Let \(I = (i_1, i_2, \dots, i_n) \in G^{\times n}\) 
    and \(J = (j_1, j_2, \dots, j_m) \in G^{\times m}\).  
    We define the \(A\)-bimodule  
    \begin{equation*}
        M(I \times J; A) \defeq \{(n \times m)\text{-matrices whose entries are in } A\}
    \end{equation*}
    with \(G\)-grading 
    \(M(I \times J; A) = \bigoplus_{d \in G} M_d(I \times J; A)\) defined by 
    \begin{equation*}
        M_d(I \times J; A) \defeq \setin{(f_{k,l})_{k,l=1}^n \in M(I \times J; A)}{
            \begin{array}{l}
            \abs{f_{k,l}} = i_k - j_l + d \\ 
            \text{ for all } k \text{ and } l
        \end{array}}. 
    \end{equation*}
    The left and right actions by \(A\) are defined by 
    \begin{equation*}
        gF \defeq (\rho(i_k, \abs{g})gf_{k,l})_{k,l}, \quad 
        Fg \defeq (\rho(j_l, \abs{g})f_{k,l}g)_{k,l} 
    \end{equation*}
    for all \(g \in A\) and \(F = (f_{k,l})_{k,l} \in M(I \times J; A)\). 
    \par For \(I = J\), the \(A\)-bimodule \(M(I \times I; A)\), 
    simply written as \(M(I; A)\), turns out to be an \(A\)-algebra. 
    \par The group consisting of invertible matrices of \(G\)-degree 0 
    is written as 
    \begin{equation*}
        \GL_0 (I; A) \defeq \setin{F \in M_0 (I; A)}{F \text{ is invertible}}. 
    \end{equation*}
\end{definition}
\begin{definition}
    Let \(F = (f_{k,l})_{k,l} \in M(I \times J; A)\). 
    The (\textit{super})\textit{transpose} of \(F\) is the matrix 
        \(\transpose{F} = (g_{l,k})_{l,k} \in M((-J) \times (-I); A)\) with 
    \begin{equation*}
        g_{l,k} \defeq \rho(i_k, j_l-i_k) f_{k,l}
    \end{equation*}
    where \(-I \defeq (-i_1, \dots, -i_n) \in G^{\times n}\) 
        and \(-J \defeq (-j_1, \dots, -j_m) \in G^{\times m}\). 
\end{definition}
%
\par Now we introduce a \(\rho\)-commutative version of determinant 
    by \cite{Covolo2016,KobayashiNagamachi1984}. 
Let \(I = (i_1, i_2, \dots, i_n) \in G^{\times n}\). 
\(I\) is said to be \textit{even} if \(I \in (G_0)^{\times n}\), 
and \textit{odd} if \(I \in (G_1)^{\times n}\). 
Denote the set of all bijections from \(I\) to \(I\) by \(\Aut(I)\). 
\begin{definition}
    Let \(F=(f_{k,l})_{k,l} \in M_0(I;A)\). 
    For each \(1 \leq k \leq n\), let \(t_k\) be an indeterminate with \(\abs{t_k} = i_k\). 
    \begin{itemize}
        \item[(a)] Suppose that \(I\) is even. 
        \(\rhodet(F) \in A_0\) is defined as the unique element characterized by 
        \begin{align*}
            &{\phantom{{}={}}} \rhodet(F) \cdot t_1 t_2 \cdots t_n \\
            &{}= \sum_{\sigma \in \Aut(I)} 
                f_{1, \sigma(1)} t_{\sigma(1)} 
                \cdot f_{2, \sigma(2)} t_{\sigma(2)} 
                \cdot \cdots 
                \cdot f_{n, \sigma(n)} t_{\sigma(n)}
        \end{align*}
        in \(A \tensor_{\Korper} S_\rho^\bullet \langle t_1, t_2, \dots, t_n \rangle\). 
        \item[(b)] Suppose that \(I\) is odd. 
        \(\rhodet(F) \in A_0\) is defined as the unique element characterized by 
        \begin{align*}
            &{\phantom{{}={}}} \rhodet(F) \cdot t_1 \wedge t_2 \wedge \cdots \wedge t_n \\
            &{}= \sum_{\sigma \in \Aut(I)} 
                f_{1, \sigma(1)} t_{\sigma(1)} 
                \wedge f_{2, \sigma(2)} t_{\sigma(2)} 
                \wedge \cdots 
                \wedge f_{n, \sigma(n)} t_{\sigma(n)}
        \end{align*} 
        in \(A \tensor_{\Korper} \wedge_\rho^\bullet \langle t_1, t_2, \dots, t_n \rangle\). 
    \end{itemize}
    We call \(\rhodet(F)\) the \textit{\(\rho\)-determinant} or the \textit{graded determinant of \(F\)}
        in both cases.  
\end{definition}
\begin{lemma}
    Suppose that \(I\) is even or odd. 
    The map 
    \begin{equation*} 
        \rhodet \colon M_0(I;A) \to A_0
    \end{equation*} 
    defined above satisfies the following properties. 
    \begin{itemize}
        \item[(a)] If \(F \in M_0(I;A)\) is invertible, then \(\rhodet(F) \in (A_0)^\times\). 
        \item[(b)] \(\rhodet(FG) = \rhodet(F)\rhodet(G)\) for all \(F, G \in M_0(I;A)\). 
        \item[(c)] Let \(F = (\mathbf{f}_{k})_{k}, G = (\mathbf{g}_{k})_{k}, 
            H = (\mathbf{h}_{k})_{k} \in M_0(I;A)\) in the row vector representation. 
            Let \(1 \leq k_0 \leq n\). 
            Suppose that \(\mathbf{f}_{k_0} + \mathbf{g}_{k_0} = \mathbf{h}_{k_0}\) 
            and \(\mathbf{f}_{k} = \mathbf{g}_{k} = \mathbf{h}_{k}\) for all \(k \neq k_0\). 
            Then, \(\rhodet(F) + \rhodet(G) = \rhodet(H)\). 
        \item[(d)] Let \(F = (\mathbf{f}_{k})_{k}, G = (\mathbf{g}_{k})_{k} \in M_0(I;A)\) 
            in the row vector representation. 
            Let \(1 \leq k_0 \leq n\) and \(c \in A_0\). 
            Suppose that \(c \mathbf{f}_{k_0} = \mathbf{g}_{k_0}\) 
            and \(\mathbf{f}_{k} = \mathbf{g}_{k}\) for all \(k \neq k_0\). 
            Then, \(c \cdot \rhodet(F) = \rhodet(G)\). 
        \item[(e)] Let \(F = (\mathbf{f}_{k})_{k} \in M_0(I;A)\) in the row vector representation. 
            Suppose that \(\mathbf{f}_{k} = \mathbf{f}_{l}\) for some \(1 \leq k < l \leq n\). 
            Then, \(\rhodet(F) = 0\). 
    \end{itemize}
\end{lemma}
\begin{proof}
    See \cite{Covolo2016,KobayashiNagamachi1984}. 
\end{proof}
In fact, you can define the \(\rho\)-determinant of a matrix of non-zero \(G\)-degree 
and find that the similar properties hold. 
However, we do not need it for discussions from here on. 
For more details, see \cite{Covolo2016,KobayashiNagamachi1984}. 
\begin{lemma}\label{lemma_rhodetexp_exptr}
    Suppose that \(I\) is even or odd. 
    Let \(F = (f_{k,l})_{k,l} = (\mathbf{f}_k)_k \in M(I;A)\). 
    Let \(\varepsilon\) be an indeterminate of \(G\)-degree \(-\abs{F}\). 
    Then, as a matrix over \(A \tensor_{\Korper} \Korper [\varepsilon]/(\varepsilon^2)\), 
        we have
    \begin{equation*}
        \rhodet(1+\varepsilon F)
        = 1 + \trace (\varepsilon F). 
    \end{equation*}
\end{lemma}
\begin{proof}
    Expand \(\rhodet(1+\varepsilon F)\) 
    by using properties of the \(\rho\)-determinant. 
    If a matrix has two or more rows 
        whose entries are all factored by \(\varepsilon\), 
        then the \(\rho\)-determinant of the matrix is 0 by the definition, i.e., 
    \begin{equation*}
        \rhodet \left( \begin{array}{c}
            \vdots \\
            \rho(i_k, \abs{\varepsilon}) \varepsilon \mathbf{f}_k \\
            \vdots \\
            \rho(i_l, \abs{\varepsilon}) \varepsilon \mathbf{f}_l \\
            \vdots 
        \end{array} \right)
        = 0. 
    \end{equation*}
    Hence, setting row vectors 
    \(\mathbf{f}_k^\prime \defeq 
        \rho(i_k, \abs{\varepsilon}) \varepsilon \mathbf{f}_k\) and 
    \(\mathbf{e}_k \defeq (\delta_{k,l})_l\), 
    we see that 
    \begin{align*}
        &\phantom{{}={}}\rhodet(1+\varepsilon F) \\
        &{}= \rhodet(1) + \rhodet \left(\begin{array}{c}
            \mathbf{f}_1^\prime \\ \mathbf{e}_2 \\ \vdots \\ \mathbf{e}_n
        \end{array}\right) 
        + \rhodet \left(\begin{array}{c}
            \mathbf{e}_1 \\ \mathbf{f}_2^\prime \\ \vdots \\ \mathbf{e}_n
        \end{array}\right) 
        + \cdots 
        + \rhodet \left(\begin{array}{c}
            \mathbf{e}_1 \\ \mathbf{e}_2 \\ \vdots \\ \mathbf{f}_n^\prime
        \end{array}\right) \\
        &{}= 1 + \rho(i_1, \abs{\varepsilon}) \varepsilon f_{1,1} 
            + \rho(i_2, \abs{\varepsilon}) \varepsilon f_{2,2} + \cdots 
            + \rho(i_n, \abs{\varepsilon}) \varepsilon f_{n,n}. 
    \end{align*}
\end{proof}
This lemma is considered as an infinitesimal version of a relation between 
the determinant of the exponential and the exponential of the trace. 

\par Then, we introduce a \(\rho\)-commutative version of Berezinian. 
Let \(I = (i_1, \dots, i_n, i_{n+1}, \dots, i_{n+m}) \in (G_0)^{\times n} \times (G_1)^{\times m}\). 
\begin{definition}\label{definition_rhober}
    Let \(F = (f_{k,l})_{k,l} \in M_0(I;A)\). 
    Suppose that \(F\) is written in the block matrix 
    \begin{equation*}
        F = \left( \begin{matrix}
            F_{00} & F_{01} \\ F_{10} & F_{11}
        \end{matrix} \right)
    \end{equation*}
    so that \(F_{00} \in M_0((i_1, \dots, i_n), A)\). 
    Define 
    \begin{equation*}
        \rhoBer(F) \defeq \left\{ \begin{array}{lr}
            \multicolumn{2}{l}{\rhodet(F_{00} - F_{01} F_{11}^{-1} F_{10}) \cdot \rhodet(F_{11})^{-1}} \\
            \multicolumn{2}{r}{\hspace{5em}\text{if both \(F_{00}\) and \(F_{11}\) are invertible},} \\
            0 & \text{otherwise}. 
        \end{array} \right.
    \end{equation*}
    We call \(\rhoBer(F)\) the \textit{\(\rho\)-Berezinian} or the \textit{graded Berezinian of \(F\)}.  
\end{definition}
\begin{proposition}\label{proposition_rhoBer} \ 
    \begin{itemize}
        \item[(i)] \(\rhoBer(FG) = (\rhoBer F)(\rhoBer G)\) for all 
        \(F, G \in \GL_0 (I; A)\).
        \item[(ii)] \(\rhoBer F = (\rhodet F_{00})\rhodet (F_{11} - F_{10} F_{00}^{-1} F_{01})^{-1}\) 
            for all \(F \in \GL_0 (I; A)\). 
        \item[(iii)] \(\rhoBer(\transpose{F}) = \rhoBer F\) for all \(F \in \GL_0 (I; A)\). 
        \item[(iv)] Suppose that \(I\) is even and \(J\) is odd. Then, we have 
            \begin{align*}
                &\phantom{{}={}} \rhoBer \left(\begin{array}{cc|cc}
                    E_{00} & 0 & F_{00} & 0 \\ 
                    E_{10} & E_{11} & F_{10} & F_{11} \\\hline
                    G_{00} & 0 & H_{00} & 0 \\
                    G_{10} & G_{11} & H_{10} & H_{11} 
                \end{array}\right) 
                \left( \eqdef \rhoBer \left(\begin{array}{cc} E & F \\ G & H \end{array} \right) \right) \\
                &{}= \rhoBer \left(\begin{array}{cc} E_{00} & F_{00} \\ G_{00} & H_{00} \end{array} \right)
                \rhoBer \left(\begin{array}{cc} E_{11} & F_{11} \\ G_{11} & H_{11} \end{array}  \right)
            \end{align*}
            for \(E \in \GL_0 (I; A)\), \(F \in M_0 (I \times J; A)\), 
                \(G \in M_0 (J \times I; A)\), and \(H \in \GL_0 (J; A)\).  
    \end{itemize}
\end{proposition}
\begin{proof}
    See \cite{KobayashiNagamachi1984} for (i). 
    (ii), (iii), and (iv) are shown by straightforward computations. 
\end{proof}
\begin{definition}
    Let \(F = (f_{k,l})_{k,l} \in M(I;A)\). 
    Define  
    \begin{equation*}
        \rhotr(F) \defeq \sum_{k=1}^{n+m} \rho(\abs{i_k}+\abs{F}, \abs{i_k}) f_{k,k}. 
    \end{equation*}
    We call \(\rhotr(F)\) the \textit{\(\rho\)-trace} or the \textit{graded trace of \(F\)}.  
\end{definition}

\begin{lemma}\label{lemma_rhoBerexp_exprhotr}
    Let \(F \in M(I;A)\). 
    Let \(\varepsilon\) be an indeterminate of \(G\)-degree \(-\abs{F}\). 
    Then, as a matrix over \(A \tensor_{\Korper} \Korper [\varepsilon]/(\varepsilon^2)\), 
        we have 
    \begin{equation*}
        \rhoBer(1 + \varepsilon F)
        = 1 + \rhotr (\varepsilon F). 
    \end{equation*}
\end{lemma}
\begin{proof}
    We use the same notation as in \thref{definition_rhober}. 
    Note that \(1 + \varepsilon F_{11}\) is invertible in 
    \(A \tensor_{\Korper} \Korper [\varepsilon]/(\varepsilon^2)\), 
    and that \((1 + \varepsilon F_{11})^{-1} = 1 - \varepsilon F_{11}\). 
    Since \(I\) is even and \(J\) is odd, we can apply 
        \thref{lemma_rhodetexp_exptr} to \(F_{00}\) and \(F_{11}\). 
    Then we obtain 
    \begin{align*}
        &\phantom{{}={}} \rhoBer(1+ \varepsilon F) \\
        &{}= \rhodet((1+\varepsilon F_{00})-
            (\varepsilon F_{01})(1 + \varepsilon F_{11})^{-1} (\varepsilon F_{10}))
            \rhodet(1 + \varepsilon F_{11})^{-1} \\
        &{}= \rhodet(1+\varepsilon F_{00} - 
            \varepsilon F_{01} (1 - \varepsilon F_{11}) \varepsilon F_{10})
            \rhodet(1 - \varepsilon F_{11}) \\
        &{}= \rhodet(1 + \varepsilon F_{00}) \rhodet (1 - \varepsilon F_{11}) \\
        &{}= (1 + \trace(\varepsilon F_{00}))(1 - \trace(\varepsilon F_{11})) \\ 
        &{}= 1 + \trace(\varepsilon F_{00}) - \trace(\varepsilon F_{11})
        = 1 + \rhotr(\varepsilon F).
    \end{align*}
\end{proof}

%% file: section_function.tex
\section{\(\rho\)-commutative functions}
In this section, we discuss the \(G\)-graded \(\rho\)-commutative algebra 
\begin{equation*}
    \rhofunc(U) \defeq \func(U) \dbrack{x^{n+1}, \dots, x^{n+m}}_\bullet 
\end{equation*}
with \(x^1, \dots, x^n\) being coordinates on a neighbourhood \(U\) of a manifold. 
This algebra corresponds to the space of local functions on a \(\rho\)-manifold 
    to be defined in the next section. 
We call \(x^1, \dots, x^n\) together with \(x^{n+1}, \dots, x^{n+m}\) 
    \textit{coordinate functions on \(U\)}. 
This algebra consists of elements of the form 
\begin{equation*}
    f = \sum_{\mathbf{w} \in \ZInteger_{\geq 0}^{\times m}} f_{\mathbf{w}} x^{\mathbf{w}}
\end{equation*}
where each \(\mathbf{w} = (w_1, w_2, \dots, w_m) \in \ZInteger_{\geq 0}^{\times m}\) is a multi-index notation with 
\begin{equation*}
    x^{\mathbf{w}} = (x^{n+1})^{w_1} (x^{n+2})^{w_2} \cdots (x^{n+m})^{w_m},  
\end{equation*}
and \(f_{\mathbf{w}} \in \func(U)\). 
Since \((x^{n+b})^2 = 0\) if \(x^{n+b}\) is odd, 
    we always assume that \(f_{\mathbf{w}} = 0\) if there is \(1 \leq b \leq m\) 
    such that \(x^{n+b}\) is odd and \(w_b \geq 2\). 
If \(f\) is assumed to be homogeneous of degree \(i \in G\), 
    then we suppose \(f_{\mathbf{w}} = 0\) unless \(|x^{\mathbf{w}}|=i\). 
\par We denote by \(\mathscr{I}\) the two-sided ideal of \(\rhofunc(U)\) 
generated by the set \(\{x^{n+1}, \allowbreak \dots, \allowbreak x^{n+m}\}\). 
Namely, \(\mathscr{I}\) consists of all formal power series 
in \(x^{n+1}, \dots, x^{n+m}\)
without terms of polynomial degree 0, 
with coefficients in \(\func(U)\). 
\begin{lemma}\label{lemma_termwise_derivation}
    Let \(D\) be a \(\rho\)-derivation on \(\rhofunc(U)\). 
    Then, for every function \(f = \sum_{\mathbf{w}} f_{\mathbf{w}} x^{\mathbf{w}} 
        \in \rhofunc(U)\), 
    the equation 
    \begin{equation*}
        D(f) = \sum_{\mathbf{w}} D(f_{\mathbf{w}} x^{\mathbf{w}})
    \end{equation*}
    holds. 
\end{lemma}
\begin{proof}
    The proof is almost same as the case of usual formal power series. 
    By definition, we have \(D(\mathscr{I}^2) \subset \mathscr{I}\) 
        and then \(D(\mathscr{I}^{k+1}) \subset \mathscr{I}^k\) 
        for all \(k \geq 1\). 
    Then \(D\) induces the well-defined derivation 
    \begin{equation*}
        D \colon \rhofunc(U) / \mathscr{I}^{k+1} \to \rhofunc(U) / \mathscr{I}^k.
    \end{equation*}
    \par Let \(f = \sum_{|x^{\mathbf{w}}|=i} f_{\mathbf{w}} x^{\mathbf{w}} \in \rhofunc(U)_i\). 
    Set \(g \defeq \sum_{|x^{\mathbf{w}}|=i} D(f_{\mathbf{w}} x^{\mathbf{w}})\). 
    This \(g\) is a well-defined formal power series 
        because \(D(\mathscr{I}^{k+1}) \subset \mathscr{I}^k\) implies that 
        each monomial of \(g\) only depends on 
        \(D(f_{\mathbf{w}} x^{\mathbf{w}})\) in finitely many \(\mathbf{w}\). 
    In \(\rhofunc(U) / \mathscr{I}^k\), we have 
    \begin{align*}
        &\phantom{{}=} D(f + \mathscr{I}^{k+1}) - (g + \mathscr{I}^k) \\
        &{}= D \left( \sum_{|x^{\mathbf{w}}|=i, |\mathbf{w}| \leq k} f_{\mathbf{w}} x^{\mathbf{w}} + \mathscr{I}^{k+1} \right) 
            - \left( \sum_{|x^{\mathbf{w}}|=i, |\mathbf{w}| \leq k} D(f_{\mathbf{w}} x^{\mathbf{w}}) + \mathscr{I}^k \right) \\
        &{}= D \left( \sum_{|x^{\mathbf{w}}|=i, |\mathbf{w}| \leq k} f_{\mathbf{w}} x^{\mathbf{w}} \right) 
            - \sum_{|x^{\mathbf{w}}|=i, |\mathbf{w}| \leq k} D(f_{\mathbf{w}} x^{\mathbf{w}})
            + \mathscr{I}^k \\
        &{}= 0 + \mathscr{I}^k, 
    \end{align*}
    where \(|\mathbf{w}| \defeq \sum_{l=1}^m w_l\) for \(\mathbf{w} \in \ZInteger_{\geq 0}^{\times m}\). 
    Thus, \(D(f) - g \in \mathscr{I}^k\) for all \(k \geq 1\). 
    Since \(\bigcap_{k \geq 1} \mathscr{I}^k = \{0\}\), 
        we conclude that \(D(f) = g\). 
\end{proof}
\thref{lemma_termwise_derivation} implies that any \(\rho\)-derivation on \(\rhofunc(U)\) 
    is uniquely determined by its action on the generators 
    \(x^1, x^2, \dots, x^{n+m}\). 
We denote by \(\ppx{x^1}, \allowbreak \ppx{x^2}, \allowbreak \dots, \allowbreak \ppx{x^{n+m}}\) the \(\rho\)-derivations on \(\rhofunc(U)\) 
    which satisfies \(\ppx{x^a}(x^b) = \delta_{a,b}\) for all \(1 \leq a, b \leq n+m\). 
%
%
\begin{lemma}\label{lemma_infinitesimal_taylor_expansion}
    Let \(f \in \rhofunc(U)\) and \(X \in \rho\Der(\rhofunc(U))\). 
    Let \(\varepsilon\) be an indeterminate with \(\abs{\varepsilon} = -\abs{X}\). 
    For a coordinate \(x=(x^a)_a\) on \(U\), 
    in \(\rhofunc(U) \tensor \Korper \dbrack{\varepsilon}/(\varepsilon^2)\), 
    it holds that 
    \begin{equation*}
        f((x^a+\varepsilon X^a)_a) = f(x) + \varepsilon \sum_{a=1}^{n+m} X^a \ppx[f]{x^a}
    \end{equation*}
    where \(X = \sum_{a} X^a \ppx{x^a}\). 
\end{lemma}
\begin{proof}
    For each \(f_{\mathbf{w}} \in \func(U)\), we have 
    \begin{align*}
        f_{\mathbf{w}} ((x^a + \varepsilon X^a)_a) 
        &{}= \sum_{\mathbf{v} \in \ZInteger_{\geq 0}^{\times n}} \frac{1}{\mathbf{v}!} \ppx[f_{\mathbf{w}}]{x^\mathbf{v}} 
        (\varepsilon X^1)^{v_1} (\varepsilon X^2)^{v_2} \cdots (\varepsilon X^n)^{v_n} \\
        &{}= f_{\mathbf{w}} + \varepsilon \sum_{a=1}^n X^a \ppx[f_{\mathbf{w}}]{x^a}. 
    \end{align*}
    For each monomial \(x^{\mathbf{w}}\), we have 
    \begin{align*}
        &\phantom{{}=} (x^{n+1} + \varepsilon X^{n+1})^{w_1} \cdots (x^{n+m} + \varepsilon X^{n+m})^{w_m} \\
        &{}= x^{n+1} \cdots x^{n+m} + \sum_{b=1}^m x^{n+1} \cdots x^{n+b-1} \varepsilon X^{n+b} x^{n+b+1} \cdots x^{n+m} \\
        &{}= x^{n+1} \cdots x^{n+m} + \varepsilon \sum_{b=1}^m X^b \ppx{x^{n+b}} (x^{n+1} \cdots x^{n+m}). 
    \end{align*}
    Therefore, 
    \begin{align*}
        &\phantom{{}=} f((x^a+\varepsilon X^a)_a) \\
        &{}= \sum_{\mathbf{w}} \left( f_{\mathbf{w}} + \varepsilon \sum_{a=1}^n X^a \ppx[f_{\mathbf{w}}]{x^a} \right) \\
        &\phantom{{}=MMM} \cdot \left( x^{n+1} \cdots x^{n+m} + \varepsilon \sum_{b=1}^m X^{n+b} \ppx{x^{n+b}}(x^{n+1} \cdots x^{n+m})\right) \\
        &{}= \sum_{\mathbf{w}} \left( f_{\mathbf{w}} x^{n+1} \cdots x^{n+m} 
            + \varepsilon \sum_{a=1}^{n+m} X^a \ppx{x^a} (f_{\mathbf{w}} x^{n+1} \cdots x^{n+m}) \right)  \\
        &{}= f + \varepsilon \sum_{a=1}^{n+m} X^a \ppx[f]{x^a}. 
    \end{align*}
\end{proof}
\par Now, we prepare the exponential function and the logarithmic function. 
They are defined only for functions of \(G\)-degree 0. 
However, they satisfy desired properties and differential equations. 
\begin{definition}\label{definition_exponential}
    For \(f = f_0 + f_1 \in \rhofunc(U)_0\) 
    with \(f_0 \in \func(U)\) and \(f_1 \in \mathscr{I}\), 
    define 
    \begin{equation*}
        \exp f \defeq \exp f_0 \cdot \sum_{k=0}^\infty \frac{1}{k!} (f_1)^k.
    \end{equation*}
\end{definition}
\begin{definition}
    Let \(f \in \rhofunc(U)_0^\times\). 
    Set \(f \eqdef f_0(1+h)\) where \(f_0 \in \func(U)^\times\) and \(h \in \mathscr{I}\). 
    Define 
    \begin{equation*}
        \log f \defeq \log f_0 + \sum_{k=1}^\infty \frac{(-1)^{k-1}}{k} h^k.
    \end{equation*}
    as long as \(\log f_0\) makes sense. 
\end{definition}
In the real category, \(\log f\) is defined if \(f_0\) values in strictly positive number. 
In the complex category, \(\log f\) is defined if 
\(f_0\) values in a simply connected domain in \(\Complex\), 
and if \(\log f_0\) is one branch of logarithm. 
\begin{proposition}\label{proposition_properties_of_exponantial}
    For all \(f, g \in \rhofunc(U)_0\), 
    the following properties hold. 
    \begin{itemize}
        \item[(i)] \(\exp(f + g) = \exp f \cdot \exp g\).
        \item[(ii)] \(\displaystyle \ppx{x^a} (\exp f) = \ppx[f]{x^a} \exp f \) \ 
            for every \(1 \leq a \leq m+n\). 
    \end{itemize}
\end{proposition}
\begin{proof}
    We use the same notation in \thref{definition_exponential}. 
    Since both \(f_1\) and \(g_1\) have \(G\)-degree 0, 
        \(\Korper \dbrack{f_1, g_1}\) is regarded as 
        a usual commutative formal power series with two indeterminates. 
    For (i), we have 
    \begin{align*}
        \exp f \cdot \exp g 
        &{}= \exp(f_0) \exp(g_0) 
            \left(\sum_{k=0}^\infty \frac{1}{k!} f_1^k \right)
            \left(\sum_{k=0}^\infty \frac{1}{k!} g_1^k \right) \\
        &{}= \exp(f_0 + g_0) \left(\sum_{k=0}^\infty \frac{1}{k!} (f_1 + g_1)^k \right) \\
        &{}= \exp (f + g). 
    \end{align*}
    For (ii), we have 
    \begin{align*}
        &\phantom{{}=} \ppx{x^a} (\exp f) \\
        &{}= (\exp f_0) \ppx[f_0]{x^a} \sum_{k=0}^\infty \frac{1}{k!}(f_1)^k
        + (\exp f_0) \sum_{k=1}^\infty \frac{1}{(k-1)!} \ppx[f_1]{x^a} (f_1)^{k-1} \\
        &{}= (\exp f_0) \left( \ppx[f_0]{x^a} + \ppx[f_1]{x^a} \right) \sum_{k=0}^\infty \frac{1}{k!}(f_1)^k 
        = \ppx[f]{x^a} \exp f. 
    \end{align*}
\end{proof}
\begin{proposition}
    For all \(f, g \in \rhofunc(U)_0^\times\), 
    the following properties hold as long as both sides make sense. 
    \begin{itemize}
        \item[(i)] \(\log(fg) = \log f + \log g\).
        \item[(ii)] \(\displaystyle \ppx{x^a} (\log f) = \ppx[f]{x^a} \frac{1}{f}\) \ 
            for every \(1 \leq a \leq m+n\). 
    \end{itemize}
\end{proposition}
\begin{proof}
    The proof is similar to \thref{proposition_properties_of_exponantial}. 
\end{proof}

%% file: section_almost_commutative_manifolds.tex
\section{\(\rho\)-manifolds}
In this section, we define \(\rho\)-manifolds 
by the analogy with supermanifolds or other graded manifolds 
(e.g. \cite{Bartocci1991,Voronov2019}). 
\begin{definition}
    Let \(G\) be an abelian group and \(\rho\) be a commutation factor on \(G\). 
    A \textit{\(\rho\)-manifold} is a ringed space \((M, \rhofunc)\) such that 
    \begin{itemize}
        \item[(i)] its underlying space \(M\) is a connected \(n\)-dimensional manifold, 
        \item[(ii)] there is an open covering \(\{U_\alpha\}_\alpha\) of \(M\) 
        which has a smooth coordinate \(\{(x_\alpha^1, x_\alpha^2, \dots, x_\alpha^n)\}_\alpha\)
        and a family 
        \begin{equation*}
            \{(x_\alpha^{n+1}, x_\alpha^{n+2}, \dots, x_\alpha^{n+m}) 
            \in (\rhofunc (U_\alpha))^{\times m}\}_\alpha
        \end{equation*}
        with 
        \begin{equation*}
            \rhofunc (U) \isom \func (U) \dbrack{x_\alpha^{n+1}, x_\alpha^{n+2}, \dots, x_\alpha^{n+m}}_\bullet. 
        \end{equation*}
    \end{itemize}
    We call each open set \(U_\alpha\) and tuple of 
    local functions \((x_\alpha^1, x_\alpha^2, \dots, x_\alpha^{n+m})\),  
    a \textit{coordinate neighbourhood} and its \textit{coordinate functions}, respectively. 
\end{definition}
\begin{remark}
    Since \(M\) is connected, a tuple of degrees 
    \begin{equation*}
        (|x_\alpha^{n+1}|, |x_\alpha^{n+2}|, \dots, |x_\alpha^{n+m}|) \in G^{\times m}
    \end{equation*}
    is constant on indices \(\alpha\) up to permutations of components. 
    Hence, we may replace \(G\) with its subgroup finitely generated 
    by \(|x_\alpha^{n+1}|, \allowbreak |x_\alpha^{n+2}|, \allowbreak 
        \dots, \allowbreak |x_\alpha^{n+m}|\) 
    for some index \(\alpha\). 
    We choose an index \(\alpha\) and assume that all coordinate functions \(y = (y^1, \dots, y^{m+n})\) 
        satisfy \(I \defeq (|x_\alpha^1|, \dots, |x_\alpha^{n+m}|) = (|y^1|, \dots, |y^{n+m}|)\). 
\end{remark}
\begin{definition}
    For a \(\rho\)-manifold \((M, \rhofunc)\), 
    we define 
    \begin{equation*}
        \Vect_\rho(M) \defeq \rho\Der(\rhofunc(M)),
    \end{equation*}
    and we call its element  
    \textit{a vector field on \((M, \rhofunc)\)}. 
\end{definition}

%
\par On the intersection of coordinates \(x=(x^a)_a\) on \(U\) and \(y=(y^a)_a\) on \(V\), 
the \textit{Jacobian matrix} \(J_{xy}\) is defined by 
\begin{equation*}
    J_{xy} \defeq \left( \ppx[y^a]{x^b} \right)_{a,b}
    = \left(\begin{array}{ccc}
        \ppx[y^1]{x^1} & \cdots & \ppx[y^1]{x^{n+m}} \\ 
        \vdots & \ddots & \vdots \\
        \ppx[y^{n+m}]{x^1} & \cdots & \ppx[y^{n+m}]{x^{n+m}}
    \end{array}
    \right)
    \in \GL_0 (I; \rhofunc(U \cap V)). 
\end{equation*}
\begin{proposition}
    For a \(\rho\)-manifold \(M\) and its coordinate \((x^a)_a\) on \(U\), 
    there exist the unique \(\rho\)-derivations 
    \(\ppx{x^1}, \dots, \ppx{x^{n+m}}\) on \(\rhofunc(U)\) with 
    \begin{itemize}
        \item[(i)] \(\ppx{x^a}(x^b) = \delta_{a,b}\) for all \(1 \leq a, b \leq m+n\), 
        \item[(ii)] 
            \(\displaystyle
                \Vect_\rho(U) \isom \left\langle \ppx{x^1}, \dots, \ppx{x^{n+m}} 
                    \right\rangle
            \) as \(\rhofunc(U)\)-modules,
        \item[(iii)] 
            as a \(\Korper\)-linear map \(\rhofunc(U) \to \rhofunc(U)\),
            \begin{equation*}
                \ppx{x^a} \ppx{x^b} = \rho(-|x^a|, -|x^b|) \ppx{x^b} \ppx{x^a}
                \quad (1 \leq a, b \leq n+m), 
            \end{equation*}
        \item[(iv)] on the intersection of \(U\) with a coordinate \(x=(x^a)_a\) 
            and \(V\) with another coordinate \(y=(y^a)_a\), 
            \begin{equation*}
                \ppx{x^b} = \sum_{a=1}^{n+m} \ppx[y^a]{x^b} \ppx{y^a}
                \quad (1 \leq b \leq n+m). 
            \end{equation*}
    \end{itemize}
\end{proposition}
\begin{proof}
    We define \(\ppx{x^a}\) by extending (i) as a \(\rho\)-derivation on \(\rhofunc(U)\). 
    \par To see (ii), let \(X \in \Vect_\rho (U)\) and set 
        \(\widetilde{X} \defeq \sum_{a=1}^{n+m} X(x^a) \ppx{x^a}\). 
    Clearly, \(X(x^a) =\widetilde{X}(x^a)\) and \(X((x^a)^k) = \widetilde{X}((x^a)^k)\). 
    From the result for derivations on \(\func(M)\), 
        it also holds that \(X(g) = \widetilde{X}(g)\) for \(g \in \func(U)\). 
    For any function \(f = \sum_{\mathbf{w}} f_{\mathbf{w}} x^{\mathbf{w}}\), we have 
    \begin{align*}
        &\phantom{{}=} (X-\widetilde{X})(f) \\
        &{}= \sum_{\mathbf{w}} (X-\widetilde{X})(f_{\mathbf{w}}) x^{\mathbf{w}} \\
        &\phantom{{}=} + \sum_{\mathbf{w}} f_{\mathbf{w}} \sum_{b=1}^m \rho_{\mathbf{w}, b} (x^{n+1})^{w_1} \cdots 
            (X-\widetilde{X})((x^{n+b})^{w_b}) \cdots (x^{n+m})^{w_m} \\ 
        &{}= 0 
    \end{align*}
    where \(\rho_{\mathbf{w}, b}\) is some constant determined from 
        the commutation between the derivation \(X-\widetilde{X}\) and \((x^{n+c})^{w_c}\) for \(c < b\). 
    \par (iii) is proved by a direct computation. It is enough to show for \(n+1 \leq a < b \leq n+m\). 
    \par For (iv), let \(f = \sum_{\mathbf{w}} f_{\mathbf{w}} y^{\mathbf{w}}\). 
    It holds that 
    \begin{equation*}
        \sum_{a=1}^n \ppx[y^a]{x^b} \ppx[f_{\mathbf{w}}]{y^a} = \ppx[f_{\mathbf{w}}]{x^b}
    \end{equation*}
    because if \(1 \leq b \leq n\) then it is just a chain rule for the underlying manifold \(M\), 
    and if \(n+1 \leq b \leq n+m\) then both sides are 0. 
    It also holds that 
    \begin{align*}
        &\phantom{{}=} \sum_{c=1}^m \ppx[y^{n+c}]{x^b} \ppx[y^{\mathbf{w}}]{y^{n+c}} \\
        &{}= \sum_{c=1}^m \rho\left( -|x^b|, \sum_{c^\prime = 1}^{a-1} w_{c^\prime} |y^{n+c^\prime}|\right) 
            (y^{n+1})^{w_1} \cdots (y^{n+c-1})^{w_{c-1}} \\ 
        &\phantom{{}=MMMMMM} \cdot w_c \ppx[y^{n+c}]{x^b} (y^{n+c})^{w_c-1} \cdot 
            (y^{n+c+1})^{w_{c+1}} \cdots (y^{n+m})^{w_m} \\
        &{}= \ppx[y^{\mathbf{w}}]{x^b}.
    \end{align*}
    Therefore, 
    \begin{align*}
        \sum_{a=1}^{n+m} \ppx[y^a]{x^b}\ppx{y^a} (f) 
        &{}= \sum_{\mathbf{w}} \left( \sum_{a=1}^{n+m} \ppx[y^a]{x^b} \ppx[f_{\mathbf{w}}]{y^a} y^{\mathbf{w}}
            + \sum_{a=1}^{n+m} \ppx[y^a]{x^b} f_{\mathbf{w}} \ppx[y^{\mathbf{w}}]{y^a} \right) \\
        &{}= \sum_{\mathbf{w}} \left( \sum_{a=1}^{n} \ppx[y^a]{x^b} \ppx[f_{\mathbf{w}}]{y^a} y^{\mathbf{w}}
            + \sum_{a=n+1}^{n+m} \ppx[y^a]{x^b} f_{\mathbf{w}} \ppx[y^{\mathbf{w}}]{y^a} \right) \\
        &{}= \sum_{\mathbf{w}} \left(\ppx[f_{\mathbf{w}}]{x^b} y^{\mathbf{w}} + f_{\mathbf{w}} \ppx[y^{\mathbf{w}}]{x^b} \right) 
        = \ppx[f]{x^b}. 
    \end{align*}
\end{proof}
\begin{definition}
    A \textit{vector bundle \(E\) on a \(\rho\)-manifold \((M, \rhofunc)\) of rank \(r\)} is 
    a locally free \(\rhofunc\)-module \(\Gamma(-, E)\) of rank \(r\). 
    In other words, 
        the sheaf \(\Gamma(-, E)\) satisfies the following conditions. 
    \begin{itemize}
        \item[(i)] There is a tuple \(I = (i_1, \dots, i_r) \in G^{\times r}\). 
        \item[(ii)] There is an open covering \(\{V_\alpha\}_\alpha\) of \(M\) 
        which forms an atlas \(\{(x_\alpha^1, \allowbreak \dots, \allowbreak x_\alpha^{n+m})\}_\alpha\) 
            on \((M, \rhofunc)\) and satisfies 
            \(\Gamma(V_\alpha, E) \isom \langle e^\alpha_1, \dots, e^\alpha_r \rangle\) 
            as \(\rhofunc(V_\alpha)\)-modules with \((|e^\alpha_1|, \dots, |e^\alpha_r|) = -I\). 
        \item[(iii)] On the intersection 
        \(V_{\alpha\beta} \defeq V_\alpha \cap V_\beta\), 
        there exists a matrix 
        \(g_{\alpha\beta} \in GL_0 (I; \rhofunc(V_{\alpha\beta}))\), 
        called the \textit{transition functions} or the \textit{transition rules}, 
        such that the transformation between \((e^\alpha_k)_k\) and \((e^\beta_l)_l\) 
        has the matrix representation
        \begin{equation*}
            (e^\alpha_1 |_{V_{\alpha\beta}} \ \cdots \ e^\alpha_r |_{V_{\alpha\beta}}) 
            = (e^\beta_1 |_{V_{\alpha\beta}} \ \cdots \ e^\beta_r |_{V_{\alpha\beta}}) \ g_{\alpha\beta} \ .
        \end{equation*}
        Moreover, the family of transition functions 
        \(\{g_{\alpha\beta}\}_{V_{\alpha\beta} \neq \varnothing}\) 
        satisfies the cocycle condition 
        \begin{equation*}
            \begin{array}{cl}
                g_{\alpha\alpha} = 1 & \text{ for all } \alpha, \text{ and }\\
                g_{\alpha\beta} g_{\beta\gamma} g_{\gamma\alpha} = 1 
                & \text{ for all } \alpha, \beta, \gamma \text{ with } 
                V_\alpha \cap V_\beta \cap V_\gamma \neq \varnothing.
            \end{array}
        \end{equation*}
    \end{itemize}
    In this case, setting the dual of \(e^\alpha_k\) as \(\xi_\alpha^k\) for each \(k\), 
        the sheaf \(\mathscr{E}\) defined by 
    \begin{equation*}
        \mathscr{E} (V_\alpha) \defeq \func(V_\alpha) \dbrack{x_\alpha^{n+1}, \dots, x_\alpha^{n+m}, 
            \xi_\alpha^1, \dots, \xi_\alpha^r}_\bullet
    \end{equation*}
    forms a \(\rho\)-manifold. 
    We call \(\xi_\alpha^1, \dots, \xi_\alpha^r\) the \textit{linear coordinates} on \(E\). 
    They are transformed as 
        \begin{equation*}
            \left(
            \begin{matrix}
                \xi_\beta^1 \big|_{V_{\alpha\beta}} \\ \vdots \\ \xi_\beta^r \big|_{V_{\alpha\beta}}
            \end{matrix}
            \right) 
            = g_{\alpha\beta} 
            \left(
            \begin{matrix}
                \xi_\alpha^1 \big|_{V_{\alpha\beta}} \\ \vdots \\ \xi_\alpha^r \big|_{V_{\alpha\beta}}
            \end{matrix}
            \right) .
        \end{equation*}
\end{definition}
\begin{example}\label{example_tangent_bundle}
    Let \((M, \rhofunc)\) be a \(\rho\)-manifold. 
    The \textit{tangent bundle of \(M\)} is 
    a vector bundle \(TM\) over \(M\) defined by
    the linear coordinates \((\partial_x^1)^\ast, \dots, (\partial_x^{n+m})^\ast\) 
    for each coordinate \((x^a)_a\) on a neighbourhood \(U\) 
    with \(\abs{(\partial_x^a)^\ast} = |x^a|\) for \(1 \leq a \leq n+m\), 
    and by the transition functions being the Jacobian matrices.  
    On the intersection of coordinates \((x^a, (\partial_x^a)^\ast)_a\) 
    and \((y^a, (\partial_y^a)^\ast)_a\), 
    the linear coordinates \(((\partial_x^a)^\ast)_a\) and \(((\partial_y^a)^\ast)_a\) 
    are transformed as 
    \begin{equation*}
        (\partial_y^a)^\ast = \sum_{b=1}^{n+m} (\partial_x^b)^\ast \ppx[y^a]{x^b}. 
    \end{equation*}
\end{example}
\begin{example}\label{example_cotangent_bundle}
    Similarly, 
    the \textit{cotangent bundle \(T^\ast M\) of \(M\)} is defined by 
    the linear coordinates \(p_1, \dots, p_{n+m}\) 
    with \(|p_a| = -|x^a|\),
    and by the transition functions being the transposition of the Jacobian matrices. 
    The linear coordinates are transformed as 
    \begin{equation*}
       q_a = \sum_{b=1}^{n+m} \ppx[x^b]{y^a} p_b 
    \end{equation*}
    from \((x^a, p_a)_a\) to \((y^a, q_a)_a\). 
\end{example}
We introduce two different types of the degree shift of 
a vector bundle. 
\begin{definition}
    Let \(E\) be a vector bundle of rank \(r\) 
        on a \(\rho\)-manifold \((M, \rhofunc)\) characterized by 
    the linear coordinates \(\xi_\alpha^1, \dots, \xi_\alpha^r\) 
    and the transition functions \(g_{\alpha\beta} = ((g_{\alpha\beta})^a_b)_{a,b}\) with  
    \begin{equation*}
        \xi_\beta^a = \sum_{a=1}^r (g_{\alpha\beta})^a_b \xi_\alpha^b
        \quad ((g_{\alpha\beta})^a_b \in \rhofunc(U_{\alpha\beta}))
        .
    \end{equation*}
    \begin{itemize}
        \item[(a)] The \textit{shift of \(E\)} or the \textit{shifted \(E\)} is a vector bundle \(\Pi E\) 
            over the \(\rho^\prime\)-manifold \((M, \func_{\rho^\prime}) = (M, \rhofunc)\)
            by the linear coordinates \((\xi^\prime)_\alpha^1, \dots, (\xi^\prime)_\alpha^r\) 
            and the transition rules 
            \begin{equation}
                (\xi^\prime)_\beta^a = \sum_{a=1}^r (g_{\alpha\beta})^a_b (\xi^\prime)_\alpha^b
                \label{equation_shifted_transition}
            \end{equation}
            with \(G^\prime\)-degrees \(|(\xi^\prime)_\alpha^a|^\prime \defeq (1, |\xi_\alpha^a|)\) 
            for \(1 \leq a \leq r\). 
        \item[(b)] Let \(i \in G\). 
        The \textit{degree \(i\)-shift of \(E\)} is a vector bundle \([-i]E\) 
        over the \(\rho\)-manifold \((M, \rhofunc)\)
        by the linear coordinates \((\xi^\prime)_\alpha^1, \dots, (\xi^\prime)_\alpha^r\) 
        and the transition rules 
        \eqref{equation_shifted_transition}
        with \(G\)-degrees \(|(\xi^\prime)_\alpha^a| \defeq |\xi_\alpha^a| - i\) 
        for \(1 \leq a \leq r\). 
    \end{itemize}
\end{definition}
The double shift \(\Pi \Pi E\) is naturally isomorphic to 
the original vector bundle \(E\) by definition, but
the double degree \(i\)-shift \([-i][-i]E = [-2i]E\) is not in general. 
\begin{example}\label{example_shifted_tangent_bundle}
    The shifted tangent bundle \(\Pi TM\) of \((M, \rhofunc)\) is described by 
    the linear coordinates \(dx^1, \dots, dx^{n+m}\) 
    with \(|dx^a|^\prime = (1, |x^a|)\) and the transition rule 
    \begin{equation*}
        dy^a = \sum_{b=1}^{n+m} dx^b \ppx[y^a]{x^b} 
    \end{equation*}
    from \((x^a, dx^a)_a\) to \((y^a, dy^a)_a\). 
    \par The sheaf of functions on \(\Pi TM\) is called 
        the \textit{de Rham complex of \((M, \rhofunc)\)}
        and denoted by \(\Omega^\bullet_M\) or simply \(\Omega^\bullet\). 
    Functions on \(\Pi TM\) are called \textit{differential forms on \((M, \rhofunc)\)}. 
    \par This is the analogy of the de Rham complex of non-graded manifolds. 
    The fundamental operations on differential forms are also defined like
    the \textit{de Rham differential} 
    \begin{equation*}
        d \defeq \sum_{a} dx^a \ppx{x^a}, 
    \end{equation*}
    the \textit{Lie derivative along \(X \in \Vect_\rho(M)\) on \(\Omega^\bullet(M)\)} 
    \begin{equation*}
        L_X \defeq \sum_{a} \left( dx^a \frac{\partial X^b}{\partial x^a} \ppx{dx^b} 
            + X^a \ppx{x^a} \right) , 
    \end{equation*}
    and \textit{the interior product of \(X \in \Vect_\rho(M)\)} 
    \begin{equation*}
        i_X \defeq \sum_{a} X^a \ppx{dx^a}.
    \end{equation*}
    Here we used a local description \(X = \sum_{a} X^a \ppx{x^a}\) 
        of a vector field \(X \in \Vect_\rho(M)\). 
    The Cartan identities in this version 
    \begin{equation*}
        [L_X, L_Y]_{\rho^\prime} = L_{[X,Y]_\rho}, 
        \quad [d, L_X]_{\rho^\prime} = 0, \quad \text{etc.}
    \end{equation*}
    can be shown \cite{Ciupala2005}. 
\end{example}
\begin{example}\label{example_shifted_cotangent_bundle}
    The shifted cotangent bundle \(\Pi T^\ast M\) is described by  
        the linear coordinates \(x_1^\ast, \dots, x_{n+m}^\ast\) 
    with \(|x_a^\ast|^\prime = (1, -|x^a|)\)
    and the transition rule 
    \begin{equation*}
        y_a^\ast = \sum_{b=1}^{n+m} \ppx[x^b]{y^a} x_b^\ast 
    \end{equation*}
    from \((x^a, x_a^\ast)_a\) to \((y^a, y_a^\ast)_a\). 
\end{example}
\begin{example}\label{example_i_shifted_cotangent_bundle}
    Consider the degree \(i\)-shift of the cotangent bundle \(T^\ast M\) 
    using the notation in \thref{example_shifted_cotangent_bundle}. 
    The linear coordinates on \([-i]T^\ast M\) are \(x_1^\ast, \dots, x_{n+m}^\ast\) 
        with \(G\)-gradings \(|x_a^\ast| = -|x^a|-i\). 
    The \textit{degree-\(i\) Schouten bracket} on \(\rhofunc([-i]T^\ast M)\) 
        is defined by 
    \begin{equation}
        \dbrack{f, g} \defeq \sum_{a=i}^{n+m} 
            \left( \rho(|f|+|x^a|+i, |x^a|+i) \ppx[f]{x_a^\ast} \ppx[g]{x^a}
            - \rho(|x^a|, |f|+i) \ppx[f]{x^a} \ppx[g]{x_a^\ast} \right)
    \end{equation}
    locally on a coordinate \((x^a, x_a^\ast)_a\). 
    The degree-0 Schouten bracket on \(T^\ast M\) corresponds to the Poisson bracket. 
    \begin{lemma}\label{lemma_i_Schouten_bracket_coordinate}
        Take a coordinate \((x^a, x^\ast_a)_a\) on \([-i]T^\ast M\) and rewrite it as 
            \begin{equation*}
                (x^1, \dots, x^{n+m}, x_1^\ast, \dots, x_{n+m}^\ast) \eqdef (z^1, \dots, z^{2(n+m)}).
            \end{equation*}
        Then, we have 
        \begin{equation*}
            \dbrack{f, g} = \sum_{a, b=1}^{2(n+m)} \rho(|z^a|, |f| - |z^a|) 
                \ppx[f]{z^a} \dbrack{z^a, z^b} \ppx[g]{z^b}. 
        \end{equation*}
    \end{lemma}
    \begin{proof}
        Note that \(\dbrack{x_a^\ast, x^b} = \delta_{a,b}\) and 
        \(\dbrack{x^a, x_b^\ast} = -\rho(|x^a|, |x^a| + i)\delta_{a,b}\) 
        for \(1 \leq a, b \leq n+m\) by the definition. 
        Hence, we have 
        \begin{align*}
            &\phantom{{}=} \sum_{a, b=1}^{2(n+m)} \rho(|z^a|, |f| - |z^a|) 
                \ppx[f]{z^a} \dbrack{z^a, z^b} \ppx[g]{z^b} \\
            &{}=\sum_{a=1}^{n+m} \rho(|x_a^\ast|, |f|-|x_a^\ast|) \ppx[f]{x_a^\ast} 
                \dbrack{x_a^\ast, x^a} \ppx[g]{x^a} \\
            &\phantom{{}=}+ \sum_{a=1}^{n+m} \rho(|x^a|, |f|-|x^a|) \ppx[f]{x^a} 
                \dbrack{x^a, x_a^\ast} \ppx[g]{x_a^\ast} \\
            &{}= \sum_{a=1}^{n+m} \rho(-|x^a|-i, |f|+|x^a|+i) 
                    \ppx[f]{x_a^\ast} \ppx[g]{x^a} \\
            &\phantom{{}=}+ \sum_{a=1}^{n+m} \rho(|x^a|, |f|-|x^a|) \rho(|x^a|, |x^a|+i) 
                    \ppx[f]{x^a} \ppx[g]{x_a^\ast} \\
            &{}= \dbrack{f, g}. 
        \end{align*}
    \end{proof}
    \begin{proposition}\label{prop_i_Schouten_bracket_properties}
        The degree-\(i\) Schouten bracket has the following properties. 
        \begin{itemize}
            \item[(i)] \(\abs{\dbrack{f, g}} = |f| + |g| + i\), 
            \item[(ii)] \(\dbrack{f, g} = -\rho(|f| + i, |g| + i) \dbrack{g, f}\), 
            \item[(iii)] \(\dbrack{f, \dbrack{g, h}} = \dbrack{\dbrack{f, g}, h}
                + \rho(|f| + i, |g| + i) \dbrack{g, \dbrack{f, h}}\), 
            \item[(iv)] \(\dbrack{f, gh} = \dbrack{f, g} h 
                + \rho(|f| + i, |g|) g \dbrack{f, h}\)
        \end{itemize}
        for all \(f, g, h \in \rhofunc([-i]T^\ast M)\). 
    \end{proposition}
    \begin{proof}
        Each property is checked by using \thref{lemma_i_Schouten_bracket_coordinate}. 
        For example, (iv) follows from 
        \begin{align*}
            \dbrack{f, gh} 
            &{}= \rho(|z^a|, |f|-|z^a|) \ppx[f]{z^a} \dbrack{z^a, z^b} \ppx[g]{z^b} h \\
            &\phantom{{}=}+ \rho(|z^a|, |f|-|z^a|) \ppx[f]{z^a} \dbrack{z^a, z^b} 
                g \cdot \rho(-|z^b|, g) \ppx[h]{x^b} \\
            &{}= \rho(|z^a|, |f|-|z^a|) \ppx[f]{z^a} \dbrack{z^a, z^b} \ppx[g]{z^b} h \\
            &\phantom{{}=}+ \rho(i+|f|, g) \rho(|z^a|, |f|-|z^a|) 
                g \ppx[f]{z^a} \dbrack{z^a, z^b} \ppx[h]{x^b} \\
            &{}= \dbrack{f, g} h + \rho(|f| + i, |g|) g \dbrack{f, h}. 
        \end{align*}
    \end{proof}
\end{example}
%
\begin{definition}
    A \textit{\(\rho\)-Q-manifold} is a pair \((M, Q)\) of 
    a \(\rho\)-manifold \(M\) and a vector field \(Q \in \Vect_\rho(M)\) 
    such that the pair \((\rhofunc(M), Q)\) forms a \(\rho\)-Q-algebra. 
    In this case, \(Q\) is called a \textit{homological vector field on \(M\)}. 
\end{definition}
\begin{example}
    Any \(\rho\)-manifold with the zero vector field \(Q = 0\) 
        is obviously a \(\rho^\prime\)-Q-manifold (cf. \thref{example_trivial_rho_Q_algebra}). 
\end{example}
\begin{example}
    Let \(M\) be a \(\rho\)-manifold. 
    As we discussed in \thref{example_shifted_tangent_bundle}, 
    the de Rham differential \(d = \sum_a dx^a \ppx{x^a}\) is squared to zero. 
    In fact, 
    \begin{align*}
        d(df) &{}=  \sum_{a,b} dx^a \ppx{x^a} \left( dx^b \ppx[f]{x^b} \right)
        = \sum_{a,b} \rho(-|x^a|, |x^b|) dx^a dx^b 
            \dfrac{\partial^2 f}{\partial x^a \partial x^b} \\ 
        &{}=\sum_{b,a} \rho(-|x^b|, |x^a|) dx^b dx^a
            \dfrac{\partial^2 f}{\partial x^b \partial x^a} \\ 
        &{}=\sum_{b,a} \rho(-|x^b|, |x^a|) \cdot (-\rho(|x^b|, |x^a|)) 
            \cdot \rho(-|x^b|, -|x^a|) dx^a dx^b 
            \dfrac{\partial^2 f}{\partial x^a \partial x^b} \\
        &{}=-\sum_{b,a} \rho(-|x^a|, |x^b|) dx^a dx^b 
            \dfrac{\partial^2 f}{\partial x^a \partial x^b}
        = 0
    \end{align*}
    for all \(f \in \rhofunc(\Pi TM) = \Omega^\bullet (M)\). 
    Thus, a pair \((\Pi TM, d)\) is a \(\rho^\prime\)-Q-manifold. 
    \par Next, we assume \(M\) is a \(\rho\)-Q-manifold 
        with a homological vector field \(Q\). 
    Since the Lie derivative along \(Q\) on \(\Omega^\bullet (M)\) is squared to zero 
    because 
    \begin{equation*}
        2 L_Q \circ L_Q = [L_Q, L_Q]_{\rho^\prime} = L_{[Q, Q]_\rho} = 0, 
    \end{equation*}
    a pair \((\Pi TM, L_Q)\) also becomes a \(\rho^\prime\)-Q-manifold. 
    \par Moreover, the Cartan identity for \(\Omega^\bullet (M)\) implies that 
        \([d, L_Q]_{\rho^\prime} = 0\). 
    It follows that \([d + L_Q, d + L_Q]_{\rho^\prime} = 0\), 
    which gives another \(\rho^\prime\)-Q-manifold \((\Pi TM, d + L_Q)\). 
\end{example}
\begin{example}\label{example_Q_structure_on_i_shifted_cotangent_bundle}
    Consider the degree-\(i\) Schouten bracket 
    (\thref{example_i_shifted_cotangent_bundle}) 
    on the degree-\(i\) shift of the cotangent bundle \([-i] T^\ast M\). 
    Suppose that \(M\) has a homological vector field \(Q\). 
    A natural correspondence 
    \begin{equation*}
        \defnamelessmap{\Vect_\rho (U)}{\rhofunc([-i]T^\ast U)}{\displaystyle\ppx{x^a}}{x_a^\ast}
    \end{equation*}
    on \(U\) with a coordinate \(x = (x^a)_a\) 
    maps \(\restr{Q}{U} = \sum_a Q^a \ppx{x^a}\) 
    to a local function \(\sum_a Q^a x_a^\ast\), 
    which is patched up to some global function \(f_Q\) 
    of degree \(\abs{f_Q} = |Q| - i\)
    because the transition rules for \(\left(\ppx{x^a}\right)_a\) 
        and \((x_a^\ast)_a\) are identical. 
    This map is of degree \(-i\). 
    Then, \(\widetilde{Q} \defeq \dbrack{f_Q, -}\) is a \(\rho\)-derivation 
        on \(\rhofunc([-i]T^\ast M)\) of degree \(|Q|\) 
        by (iv) in \thref{prop_i_Schouten_bracket_properties}. 
    Using \thref{lemma_i_Schouten_bracket_coordinate}, we have locally 
    \begin{align*}
        &\phantom{{}=}\dbrack{f_Q, f_Q} \\
        &{}= \rho(-|x^a|-i, |Q|+|x^a|) \ppx{x_a^\ast}(Q^b x_b^\ast) \ppx{x^a}(Q^c x_c^\ast) \\
        &\phantom{{}=} - \rho(|x^a|, |Q|) \ppx{x^a} (Q^b x_b^\ast) \ppx{x_a^\ast} (Q^c x_c^\ast) \\
        &{}= Q^a \ppx[Q^c]{x^a} x_c^\ast - \rho(|x^a|, |Q|)\rho(|x^a|+i, |Q|+|x^a|) \ppx[Q^b]{x^a} x_b^\ast Q^a \\
        &{}= 2 Q^a \ppx[Q^c]{x^a} x_c^\ast = 0.
    \end{align*}
    The last equality is just an interpretation of the condition \([Q,Q]_\rho =0\). 
    Hence, for all \(f \in \rhofunc([-i]T^\ast M)\), we have 
    \begin{align*}
        &\phantom{{}=}\widetilde{Q} (\widetilde{Q} f) 
        = \dbrack{f_Q, \dbrack{f_Q, f}} \\
        &{}= \dbrack{\dbrack{f_Q, f_Q}, f} 
            + \rho(|f_Q| + i, |f_Q| + i) \dbrack{f_Q, \dbrack{f_Q, f}} \\
        &{}= -\dbrack{f_Q, \dbrack{f_Q, f}}, 
    \end{align*}
    which yields that \(\widetilde{Q} \circ \widetilde{Q} = 0\). 
    Therefore, \(([-i]T^\ast M, \widetilde{Q})\) is  
        a \(\rho\)-Q-manifold. 
\end{example}
\begin{example}\label{example_Q_structure_on_noncommutative_tori}
    A noncommutative torus (\thref{example_noncommutative_torus_algebra}) is considered as 
    a \(\rho\)-manifold \(M\) whose underlying manifold is a point. 
    A \(\rho\)-commutative algebra \(A_\Theta\) is 
        \(\ZInteger^{\times m}\)-graded and its generators are graded positively, 
        which gives 
    \begin{equation*}
        A_\Theta = C^\omega (\{\ast\}, \Complex) \dbrack{u^1, \dots, u^m}_\bullet
            = \Complex [u^1, \dots, u^m]. 
    \end{equation*}
    Its BRST quantization corresponds to the \(\rho^\prime\)-commutative algebra 
        \(C^\omega_{\rho^\prime} (\Pi\gee \times M, \Complex)\)
        where \(\gee\) is the Lie algebra of \(m\)-dimensional torus. 
    The corresponding BRST differential \(Q\) is 
    \begin{equation*}
        Q = -\sum_{a=1}^m 2 \pi \sqrt{-1} \eta^a u^a \ppx{u^a}
            \in \Vect_{\rho^\prime} (\Pi\gee \times M)
    \end{equation*}
    where each \(\eta^a \in C^\omega_{\rho^\prime}(\Pi\gee, \Complex)\) is a coordinate function with 
    \begin{equation*}
        |\eta^a|^\prime = (1, (0, \dots, 0)) \in \ZInteger \times \ZInteger^{\times m}. 
    \end{equation*}
    This \(Q\) becomes a homological vector field on \(\Pi\gee \times M\). 
\end{example}

%% file: section_volume.tex
\section{Volume forms on a \(\rho\)-manifold}
For a while, let \(M\) be a \(G\)-graded \(\rho\)-manifold. 
\begin{definition}
    Let \(M\) be a \(\rho\)-manifold. 
    The \textit{\(\rho\)-Berezinian bundle on \(M\)} is the line bundle 
    \(\rhoBer(M)\) on \(M\) characterized by 
    a local frame \(D(x)\) of degree 0 
    on a coordinate neighbourhood \(x=(x^a)_a\), 
    and the transition rule 
    \begin{equation*}
        D(y) = D(x) \rhoBer \left(\frac{\partial y^a}{\partial x^b}\right)_{a,b} 
    \end{equation*}
    between coordinate neighbourhoods \(y=(y^a)_a\) and \(x=(x^b)_b\). 
\end{definition}
Since the Jacobian matrix \(\left(\frac{\partial y^a}{\partial x^b}\right)_{a,b}\) between neighbourhoods 
clearly satisfies the cocycle condition, 
by \thref{proposition_rhoBer}, a vector bundle \(\rhoBer(M)\) is well-defined. 
\begin{definition}
    A \(\rho\)-manifold \(M\) is said to be \textit{orientable} 
    if there is a global section \(\vol \in \Gamma(M, \rhoBer(M))\) 
    whose restriction to each neighbourhood \(U_x\) with a coordinate \(x=(x^a)_a\) 
    is of a form \(\restr{\vol}{U_x} = D(x) s(x)\) 
    with \(s(x) \allowbreak \in (C_\rho^\infty (U_x))_0\) invertible. 
    \par In this case, we call this section \(\vol\) a \textit{\(\rho\)-Berezin volume form on \(M\)}
        or simply a \textit{volume form on \(M\)}. 
    We denote by \(\mathrm{Vol}(M)\) the set of all \(\rho\)-Berezin volume forms on \(M\). 
\end{definition}
\begin{definition}\label{definition_equivalent_volume_forms}
    Let \(M\) be a \(\rho\)-manifold and 
    \(\vol_1, \vol_2 \in \Gamma(M, \rhoBer(M))\) be two \(\rho\)-Berezin volume forms. 
    \(\vol_1\) and \(\vol_2\) are said to be \textit{equivalent} 
    if \(\vol_2 = \vol_1 \cdot \exp h\) for some \(h \in \rhofunc(M)_0\).  
\end{definition}
\begin{remark}\label{remark_equivalent_volume_forms}
    Let us consider the equivalence of volume forms in the non-graded case. 
    The \(\rho\)-Berezinian bundle is just the determinant line bundle \(\det M\) on \(M\). 
    A volume form on \(M\) is a non-vanishing section on \(\det M\), 
        which makes \(\det M\) trivial. 
    \par In the real category, since any positive function is 
    the exponential of some function, two volume forms are equivalent 
    if and only if they are identical up to a multiple by a positive function. 
    This equivalence relation divide \(\mathrm{Vol}(M)\) into two equivalence classes. 
    \par In the complex category, the situation is a bit more complicated. 
    For example, on a punctured complex plane \(M = \Complex^\times\), 
    we have two volume forms \(\vol_1 = 1 dz\) and \(\vol_2 = z dz\) 
    which are not equivalent. 
    Since the action of \(\func(M)^\times\) by multiplication 
        on \(\mathrm{Vol}(M)\) is transitive and free, 
        if we pick some volume form, 
        then \(\mathrm{Vol}(M)\) is identified with \(\func(M)^\times\). 
    Recall that the exponential sheaf sequence 
    \begin{equation*}
        0 \longrightarrow \ZInteger 
        \overset{2\pi\sqrt{-1}}{\longrightarrow} \func 
        \overset{\exp}{\longrightarrow} \func^\times 
        \longrightarrow 0
    \end{equation*}
    yields a long exact sequence 
    \begin{equation*}
        0 \to H^0 (M, \ZInteger) \to H^0 (M, \func) \to H^0 (M, \func^\times) 
        \overset{\delta}{\to} H^1 (M, \ZInteger) \to \cdots
        .
    \end{equation*}
    The identification \(H^0 (M, \func) = \func(M)\) 
        and \(H^0 (M, \func^\times) = \func^\times(M)\) gives 
    \begin{align*}
        \func(M)^\times / \sim 
        &{}= \func(M)^\times / \Image (\exp \colon H^0 (M, \func) \to H^0 (M, \func^\times)) \\
        &{}= \func(M)^\times / \Kernel \delta
    \end{align*}
    where \(\sim\) means the equivalence relation in \thref{definition_equivalent_volume_forms}. 
    The connecting homomorphism \(\delta\) 
    satisfies
    \begin{equation*}
        (\delta f)(\gamma) = \frac{1}{2\pi\sqrt{-1}} \int_{f \circ \gamma} \frac{dz}{z}
    \end{equation*}
    for a loop \(\gamma \in H_1 (M, \ZInteger)\). 
    Therefore, the equivalence relation \(\sim\) between volume forms means 
        they have ``the same winding number''. 
\end{remark}

\begin{example}
    If \(M\) is a point or an open subset of \(\Korper^n\), 
        then \(M\) has a trivial \(\rho\)-Berezin volume form 
        \(\vol = D(x) \cdot 1\). 
\end{example}
\begin{example}\label{example_volume_of_shifted_tangent_bundle}
    Let us consider the shifted tangent bundle \(\Pi TM\) (\thref{example_shifted_tangent_bundle}) 
        of \(M\) with coordinates \(x=(x^a)_a\). 
    To deal with the Jacobian matrix, we write the even coordinates 
        as \(x^{\even} \defeq (x^a)_{|x^a| \text{ : even}}\)
        and the odd coordinates 
        as \(x^{\odd} \defeq (x^a)_{|x^a| \text{ : odd}}\). 
    The \(\rho^\prime\)-Berezinian of the Jacobian matrix is written in block matrices as 
    \begin{align*}
        &\phantom{{}={}} \rho^\prime\Ber J_{(x,dx),(y,dy)} \\
        &{}= \rho^\prime\Ber \left( \begin{array}{cccc}
            \big(\ppx[y^{\even}]{x^{\even}}\big) & 0 & 
                \big(\ppx[y^{\even}]{x^{\odd}}\big) & 0 \\[1ex]
            \big(\ppx[dy^{\odd}]{x^{\even}}\big) & \big(\ppx[dy^{\odd}]{dx^{\odd}}\big) & 
                \big(\ppx[dy^{\odd}]{x^{\odd}}\big) & \big(\ppx[dy^{\odd}]{dx^{\even}}\big) \\[1ex]
            \big(\ppx[y^{\odd}]{x^{\even}}\big) & 0 & 
                \big(\ppx[y^{\odd}]{x^{\odd}}\big) & 0 \\[1ex]
            \big(\ppx[dy^{\even}]{x^{\even}}\big) & \big(\ppx[dy^{\even}]{dx^{\odd}}\big) & 
                \big(\ppx[dy^{\even}]{x^{\odd}}\big) & \big(\ppx[dy^{\even}]{dx^{\even}}\big) 
        \end{array} \right) \\
        &{}= \rho^\prime\Ber \left( \begin{array}{cc}
            \big(\ppx[y^{\even}]{x^{\even}}\big) & \big(\ppx[y^{\even}]{x^{\odd}}\big) \\
            \big(\ppx[y^{\odd}]{x^{\even}}\big) & \big(\ppx[y^{\odd}]{x^{\even}}\big) 
        \end{array} \right)
        \rho^\prime\Ber \left( \begin{array}{cc}
            \big(\ppx[dy^{\odd}]{dx^{\odd}}\big) & \big(\ppx[dy^{\odd}]{dx^{\even}}\big) \\
            \big(\ppx[dy^{\even}]{dx^{\odd}}\big) & \big(\ppx[dy^{\even}]{dx^{\even}}\big) 
        \end{array} \right) \\
        &{}= \rho^\prime\Ber \left( \begin{array}{cc}
            \big(\ppx[y^{\even}]{x^{\even}}\big) & \big(\ppx[y^{\even}]{x^{\odd}}\big) \\
            \big(\ppx[y^{\odd}]{x^{\even}}\big) & \big(\ppx[y^{\odd}]{x^{\even}}\big) 
        \end{array} \right)
        \rho^\prime\Ber \left( \begin{array}{cc}
            \big(\ppx[y^{\even}]{x^{\even}}\big) & \big(\ppx[y^{\even}]{x^{\odd}}\big) \\
            \big(\ppx[y^{\odd}]{x^{\even}}\big) & \big(\ppx[y^{\odd}]{x^{\even}}\big) 
        \end{array} \right)^{-1} \\
        &{}= 1. 
    \end{align*}
    Here we used \thref{proposition_rhoBer} (iv). 
    Hence, \(\Pi TM\) has a \(\rho^\prime\)-Berezin volume form \(\vol = D(x,dx) \cdot 1\). 
\end{example}
\begin{example}\label{example_volume_of_cotangent_bundle}
    For the cotangent bundle \(T^\ast M\) (\thref{example_cotangent_bundle}), we have
    \begin{align*}
        &\phantom{{}={}} \rhoBer J_{(x,p),(y,q)} \\
        &{}= \rhoBer \left( \begin{array}{cccc}
            \big(\ppx[y^{\even}]{x^{\even}}\big) & 0 & 
                \big(\ppx[y^{\even}]{x^{\odd}}\big) & 0 \\[1ex]
            \big(\ppx[q_{\even}]{x^{\even}}\big) & \big(\ppx[q_{\even}]{p_{\even}}\big) & 
                \big(\ppx[q_{\even}]{x^{\odd}}\big) & \big(\ppx[q_{\even}]{p_{\odd}}\big) \\[1ex]
            \big(\ppx[y^{\odd}]{x^{\even}}\big) & 0 & 
                \big(\ppx[y^{\odd}]{x^{\odd}}\big) & 0 \\[1ex]
            \big(\ppx[q_{\odd}]{x^{\even}}\big) & \big(\ppx[q_{\odd}]{p_{\even}}\big) & 
                \big(\ppx[q_{\odd}]{x^{\odd}}\big) & \big(\ppx[q_{\odd}]{p_{\odd}}\big) 
        \end{array} \right) \\
        &{}= \rhoBer \left( \begin{array}{cc}
            \big(\ppx[y^{\even}]{x^{\even}}\big) & \big(\ppx[y^{\even}]{x^{\odd}}\big) \\
            \big(\ppx[y^{\odd}]{x^{\even}}\big) & \big(\ppx[y^{\odd}]{x^{\even}}\big) 
        \end{array} \right)
        \rhoBer \left( \begin{array}{cc}
            \big(\ppx[q_{\even}]{p_{\even}}\big) & \big(\ppx[q_{\even}]{p_{\odd}}\big) \\
            \big(\ppx[q_{\odd}]{p_{\even}}\big) & \big(\ppx[q_{\odd}]{p_{\odd}}\big)  
        \end{array} \right) \\
        &{}= \rhoBer \left( \begin{array}{cc}
            \big(\ppx[y^{\even}]{x^{\even}}\big) & \big(\ppx[y^{\even}]{x^{\odd}}\big) \\
            \big(\ppx[y^{\odd}]{x^{\even}}\big) & \big(\ppx[y^{\odd}]{x^{\even}}\big) 
        \end{array} \right)
        \rhoBer \transpose{\left( \begin{array}{cc}
            \big(\ppx[x^{\even}]{y^{\even}}\big) & \big(\ppx[x^{\even}]{y^{\odd}}\big) \\
            \big(\ppx[x^{\odd}]{y^{\even}}\big) & \big(\ppx[x^{\odd}]{y^{\even}}\big) 
        \end{array} \right)} \\
        &{}= 1. 
    \end{align*}
    Here we used \thref{proposition_rhoBer} (iii) and (iv). 
    Hence, \(T^\ast M\) has a \(\rho\)-Berezin volume form \(\vol = D(x, p) \cdot 1\). 
\end{example}
\begin{example}
    For the tangent bundle \(TM\) (\thref{example_tangent_bundle}), we have 
    \begin{align*}
        &\phantom{{}={}} \rhoBer J_{(x,p),(y,q)} \\
        &{}= \rhoBer \left( \begin{array}{cccc}
            \big(\ppx[y^{\even}]{x^{\even}}\big) & 0 & 
                \big(\ppx[y^{\even}]{x^{\odd}}\big) & 0 \\[1ex]
            \big(\ppx[(\partial_y^{\even})^\ast]{x^{\even}}\big) & \big(\ppx[(\partial_y^{\even})^\ast]{(\partial_x^{\even})^\ast}\big) & 
                \big(\ppx[(\partial_y^{\even})^\ast]{x^{\odd}}\big) & \big(\ppx[(\partial_y^{\even})^\ast]{(\partial_x^{\odd})^\ast}\big) \\[1ex]
            \big(\ppx[y^{\odd}]{x^{\even}}\big) & 0 & 
                \big(\ppx[y^{\odd}]{x^{\odd}}\big) & 0 \\[1ex]
            \big(\ppx[(\partial_y^{\odd})^\ast]{x^{\even}}\big) & \big(\ppx[(\partial_y^{\odd})^\ast]{(\partial_x^{\even})^\ast}\big) & 
                \big(\ppx[(\partial_y^{\odd})^\ast]{x^{\odd}}\big) & \big(\ppx[(\partial_y^{\odd})^\ast]{(\partial_x^{\odd})^\ast}\big) 
        \end{array} \right) \\
        &{}= \rhoBer \left( \begin{array}{cc}
            \big(\ppx[y^{\even}]{x^{\even}}\big) & \big(\ppx[y^{\even}]{x^{\odd}}\big) \\
            \big(\ppx[y^{\odd}]{x^{\even}}\big) & \big(\ppx[y^{\odd}]{x^{\even}}\big) 
        \end{array} \right)
        \rhoBer \left( \begin{array}{cc}
            \big(\ppx[(\partial_y^{\even})^\ast]{(\partial_x^{\even})^\ast}\big) & \big(\ppx[(\partial_y^{\even})^\ast]{(\partial_x^{\odd})^\ast}\big) \\
            \big(\ppx[(\partial_y^{\odd})^\ast]{(\partial_x^{\even})^\ast}\big) & \big(\ppx[(\partial_y^{\odd})^\ast]{(\partial_x^{\odd})^\ast}\big)  
        \end{array} \right) \\
        &{}= \rhoBer \left( \begin{array}{cc}
            \big(\ppx[y^{\even}]{x^{\even}}\big) & \big(\ppx[y^{\even}]{x^{\odd}}\big) \\
            \big(\ppx[y^{\odd}]{x^{\even}}\big) & \big(\ppx[y^{\odd}]{x^{\even}}\big) 
        \end{array} \right)^2 \\
        &{}= (\rhoBer J_{xy})^2 .
    \end{align*}
    Therefore, the tangent bundle \(TM\) has a \(\rho\)-Berezin volume form 
        if the base \(\rho\)-manifold \(M\) has a \(\rho\)-Berezin volume form \(\vol = D(x)s(x)\). 
    In this case, \(\widetilde{\vol} \defeq D(x, (\partial_x)^\ast) s(x)^2\) becomes 
        a \(\rho\)-Berezin volume form on \(TM\). 
\end{example}
\begin{example}\label{example_volume_of_shifted_cotangent_bundle}
    For the shifted cotangent bundle \(\Pi T^\ast M\) (\thref{example_shifted_cotangent_bundle}), we have 
    \begin{align*}
        &\phantom{{}={}} \rho^\prime\Ber J_{(x,x^\ast),(y,y^\ast)} \\
        &{}= \rho^\prime\Ber \left( \begin{array}{cccc}
            \big(\ppx[y^{\even}]{x^{\even}}\big) & 0 & 
                \big(\ppx[y^{\even}]{x^{\odd}}\big) & 0 \\[1ex]
            \big(\ppx[y^\ast_{\odd}]{x^{\even}}\big) & \big(\ppx[y^\ast_{\odd}]{x^\ast_{\odd}}\big) & 
                \big(\ppx[y^\ast_{\odd}]{x^{\odd}}\big) & \big(\ppx[y^\ast_{\odd}]{x^\ast_{\even}}\big) \\[1ex]
            \big(\ppx[y^{\odd}]{x^{\even}}\big) & 0 & 
                \big(\ppx[y^{\odd}]{x^{\odd}}\big) & 0 \\[1ex]
            \big(\ppx[y^\ast_{\even}]{x^{\even}}\big) & \big(\ppx[y^\ast_{\even}]{x^\ast_{\odd}}\big) & 
                \big(\ppx[y^\ast_{\even}]{x^{\odd}}\big) & \big(\ppx[y^\ast_{\even}]{x^\ast_{\even}}\big) 
        \end{array} \right) \\
        &{}= \rho^\prime\Ber \left( \begin{array}{cc}
            \big(\ppx[y^{\even}]{x^{\even}}\big) & \big(\ppx[y^{\even}]{x^{\odd}}\big) \\
            \big(\ppx[y^{\odd}]{x^{\even}}\big) & \big(\ppx[y^{\odd}]{x^{\even}}\big) 
        \end{array} \right)
        \rho^\prime\Ber \left( \begin{array}{cc}
            \big(\ppx[y^\ast_{\odd}]{x^\ast_{\odd}}\big) & \big(\ppx[y^\ast_{\odd}]{x^\ast_{\even}}\big) \\
            \big(\ppx[y^\ast_{\even}]{x^\ast_{\odd}}\big) & \big(\ppx[y^\ast_{\even}]{x^\ast_{\even}}\big) 
        \end{array} \right) \\
        &{}= \rho^\prime\Ber \left( \begin{array}{cc}
            \big(\ppx[y^{\even}]{x^{\even}}\big) & \big(\ppx[y^{\even}]{x^{\odd}}\big) \\
            \big(\ppx[y^{\odd}]{x^{\even}}\big) & \big(\ppx[y^{\odd}]{x^{\even}}\big) 
        \end{array} \right)^2 \\
        &{}= (\rhoBer J_{xy})^2 .
    \end{align*}
    Therefore, the shifted cotangent bundle \(\Pi T^\ast M\) has a \(\rho\)-Berezin volume form 
        given by \(\widetilde{\vol} \defeq D(x, x^\ast) s(x)^2\)
        if the base \(\rho\)-manifold \(M\) has a \(\rho\)-Berezin volume form \(\vol = D(x)s(x)\). 
\end{example}
\begin{example}
    Consider the degree-\(i\) shifted cotangent bundle \([-i]T^\ast M\) (\thref{example_i_shifted_cotangent_bundle}). 
    If \(i\) is even, then the parity of each coordinate \(x_a^\ast\) is the same with that of \(x_a\). 
    Hence, the \(\rho\)-Berezinian of the Jacobian matrix is always 1  
        similarly to \thref{example_volume_of_cotangent_bundle}.  
        Thus \([-i]T^\ast M\) has a \(\rho\)-Berezin volume form 
        \(\vol = D((x^a, x_a^\ast)_a) \cdot 1\). 
    \par If \(i\) is odd, then the parity of each coordinate \(x_a^\ast\) is opposite to that of \(x_a\).  
    Hence, the \(\rho\)-Berezinian of the Jacobian matrix is 
        the square of the \(\rho\)-Berezinian of the Jacobian matrix
        on the base \(\rho\)-manifold 
        similarly to \thref{example_volume_of_shifted_cotangent_bundle}. 
    Consequently, \([-i]T^\ast M\) has a \(\rho\)-Berezin volume form 
        \(\widetilde{\vol} = D(x, x^\ast) s(x)^2\) 
        if \(M\) itself has a \(\rho\)-Berezin volume form \(\vol = D(x) s(x)\). 
\end{example}

\par By deforming the coordinates infinitesimally, 
we consider Lie derivative of a section of \(\rho\)-Berezinian bundle. 
\par Let \(X\) be a vector field on a \(\rho\)-manifold \(M\). 
Infinitesimal deformation of a coordinate \((x^a)_a\) on a neighbourhood \(U_x\) 
is regarded as 
\begin{equation*}
    x^a \mapsto x^a + \varepsilon X^a \eqdef y^a
\end{equation*}
where \(X = \sum_a X^a \ppx{x^a}\) is the local description of \(X\) on \(U_x\), 
and \(\varepsilon\) is an infinitesimal parameter, 
i.e., an indeterminate of degree \(-\abs{X}\) with \(\varepsilon^2 = 0\). 
The Jacobian matrix of this deformation becomes 
\begin{align*}
    \left(\ppx{x^b}(x^a + \varepsilon X^a)\right)_{a,b}
    &{}= \left( \delta_{a,b} + \rho(-\abs{x_b},\abs{\varepsilon}) \varepsilon 
        \frac{\partial X^a}{\partial x^b} \right)_{a,b} \\
    &{}= 1 + \varepsilon \cdot \left(\rho(-\abs{x_a}-\abs{x_b}, \abs{\varepsilon}) 
        \frac{\partial X^a}{\partial x^b}\right)_{a,b}.
\end{align*}
From \thref{lemma_rhoBerexp_exprhotr}, the deformation 
of the local frame \(D(x)\) becomes 
\begin{align*}
    D(y) &{}= D(x) \rhoBer \left(\ppx{x^b}(x^a + \varepsilon X^a)\right)_{a,b} \\
    &{}= D(x) \rhoBer \left( 1 + 
        \varepsilon \cdot \left(\rho(-\abs{x_a}-\abs{x_b}, \abs{\varepsilon}) 
        \frac{\partial X^a}{\partial x^b}\right)_{a,b}\right) \\
    &{}= D(x) \left(1 + \rhotr \left( \varepsilon \cdot 
        \left(\rho(-\abs{x_a}-\abs{x_b}, \abs{\varepsilon}) 
        \frac{\partial X^a}{\partial x^b}\right)_{a,b} \right)\right) \\
    &{}= D(x) \left(1 + \varepsilon \sum_{a} \rho(\abs{x_a}, \abs{x_a} + \abs{X})
        \frac{\partial X^a}{\partial x^a} \right).
\end{align*}
On the other hand, for every local function \(s(x) \in C_\rho^\infty (U_x)\), 
we have 
\begin{equation*}
    s(y) = s(x) + \varepsilon X^a \frac{\partial s}{\partial x^a} 
\end{equation*}
from \thref{lemma_infinitesimal_taylor_expansion}. 
Thus, 
\begin{align*}
    D(y) s(y) 
    &{}= D(x) \left(1 + \varepsilon \sum_{a} \rho(\abs{x_a}, \abs{x_a} + \abs{X})
        \frac{\partial X^a}{\partial x^a} \right)
        \left( s(x) + \varepsilon X^a \frac{\partial s}{\partial x^a} \right) \\
    &{}= D(x)s(x) + \varepsilon D(x) \sum_{a} \rho(\abs{x^a}, \abs{x^a} + \abs{X})
        \ppx{x^a} (X^a s). 
\end{align*}
Therefore, the \textit{Lie derivative along \(X\)} 
of a section of \(\rhoBer(M)\) is described locally by 
\begin{equation} \label{formula_L_X_D_s}
    L_X (D(x)s(x)) = D(x) \sum_{a} 
        \rho(\abs{x^a}, \abs{x^a} + \abs{X}) \ppx{x^a} (X^a s)
\end{equation}
for a coordinate \(x=(x^a)_a\) on \(U_x\) 
and a local function \(s(x) \in C_\rho^\infty (U_x)\). 

The Lie derivative of a \(\rho\)-Berezin volume form on a \(\rho\)-manifold 
is again a section of \(\rho\)-Berezinian bundle. 
The divergence of a \(\rho\)-Berezin volume form is defined as follows. 
\begin{definition}
    Let \(\vol\) be a \(\rho\)-Berezin volume form on \(M\). 
    For a vector field \(X \in \Vect_\rho(M)\), 
    define the function \(\Diverg_{\vol} X\)
    called the \textit{divergence of \(X\) with respect to \(\vol\)} by 
    \begin{equation*}
        \vol \cdot (\Diverg_{\vol} X) = L_X (\vol). 
    \end{equation*}
    In a coordinate \(x=(x^a)_a\), it is written as 
    \begin{equation*}
        \Diverg_{\vol} X = \sum_a \rho(\abs{x^a}, \abs{x^a} + \abs{X}) 
            s^{-1} \ppx{x^a} (X^a s)
    \end{equation*}
    where \(\vol \eqdef D(x)s(x)\). 
\end{definition}

\begin{proposition}\label{proposition_properties_of_divergence}
    For all \(X, Y \in \Vect_\rho(M)\), \(f, g \in C_\rho^\infty(M)\) with \(\abs{g} = 0\), 
    and a \(\rho\)-Berezin volume form \(\vol\) on \(M\), 
    the following properties hold: 
    \begin{itemize}
        \item[(i)] \(\Diverg_{\vol} (fX) = f \Diverg_{\vol} X 
            + \rho(\abs{f}, \abs{X}) X(f)\), 
        \item[(ii)] \(\Diverg_{\vol \cdot \exp(g)} X = \Diverg_{\vol} X + X(g)\), 
        \item[(iii)] \(\Diverg_{\vol} [X,Y]_\rho = X(\Diverg_{\vol} Y) 
            - \rho(\abs{X}, \abs{Y}) Y(\Diverg_{\vol} X)\). 
    \end{itemize}
\end{proposition}
\begin{proof}
    By a direct computation. 
\end{proof}

The definition of the divergence and its properties above 
coincide with the ones in the case of supermanifolds \cite{Bruce2017}.

%% file: section_modular_class.tex
\section{The modular class of a \(\rho\)-manifold}
For a \(\rho\)-Q-manifold \((M, Q)\), 
the condition \(Q^2 = 0\) implies that 
the subalgebra 
\begin{equation*}
    \rhofunc(M)_Q \defeq \bigoplus_{k \in \ZInteger} \rhofunc(M)_{k \abs{Q}}
    \subset \rhofunc(M) 
\end{equation*}
has a structure of a cochain complex by an action of \(Q\). 
We write the cohomology group of this cochain complex \((\rhofunc(M)_Q, Q)\) 
as \(H_Q^\bullet (M)\). 

\begin{lemma}
    For a \(\rho\)-Q-manifold \((M, Q)\) with a \(\rho\)-Berezin volume form \(\vol\), 
    \begin{equation*}
        Q (\Diverg_{\vol} Q) = 0. 
    \end{equation*}
\end{lemma}
\begin{proof}
    Write \(\vol = D(x)s(x)\) locally. 
    Since \(L_Q \circ L_Q = 0\), we have  
    \begin{align*}
        0 &{}= L_Q L_Q (\vol) = L_Q (D(x) s(x) \Diverg_{\vol} Q) \\
        &{}= \rho (|x^a|, \abs{Q} + |x^a|) D(x) \ppx{x^a} (Q^a s(x) \Diverg_{\vol} Q) \\
        &{}= (D(x) s(x)) (\Diverg_{\vol} Q)^2 + D(x) s(x) Q(\Diverg_{\vol} Q) \\
        &{}= \vol \cdot Q (\Diverg_{\vol} Q). 
    \end{align*}
\end{proof}

\begin{definition}
    The \textit{modular class} of a \(\rho\)-Q-manifold \((M, Q)\) 
    with a \(\rho\)-Berezin volume form \(\vol\) 
    is the cohomology class of the divergence of \(Q\), i.e. 
    \begin{equation*}
        \Modular_Q (M; \vol) \defeq [\Diverg_{\vol} Q] \in H_Q^1(M). 
    \end{equation*}
\end{definition}
\begin{lemma}
    Suppose that \((M, Q)\) is an orientable \(\rho\)-Q-manifold. 
    The modular class of \((M, Q)\) does not depend on 
    equivalent \(\rho\)-Berezin volume forms. 
    Moreover, in the real category, it holds also for 
    \(\rho\)-Berezin volume forms which are not equivalent. 
\end{lemma}
\begin{proof}
    Suppose that \(\vol_1\) and \(\vol_2\) are equivalent 
        \(\rho\)-Berezin volume forms on \(M\). 
    There is some function \(h \in \rhofunc(M)_0\) such that 
        \(\vol_2 = \vol_1 \cdot \exp h\). 
    From \thref{proposition_properties_of_divergence} (ii), 
        it holds that \(\Diverg_{\vol_2} Q = \Diverg_{\vol_1} Q + Q(h)\). 
    \par In the real category, 
        for any pair of volume forms \(\vol_1\) and \(\vol_2\), 
        there is some invertible function 
        \(f = \sum_{\mathbf{w}} f_{\mathbf{w}} x^{\mathbf{w}} \in \rhofunc(M)_0\) 
        with \(\vol_2 = \vol_1 \cdot f\). 
    Let \(f_0 \in \func(M)\) be the term in \(f\) 
        which has no \(G\)-graded indeterminates. 
    The invertibility of \(f\) implies that \(f_0\) is always positive 
        or always negative. 
    Hence, we have \(\vol_2 = \vol_1 \cdot (\pm \exp (\log \abs{f}))\) and 
    \begin{equation*}
        \Diverg_{\vol_2} Q = \Diverg_{\vol_1} Q \pm Q(\log \abs{f}). 
    \end{equation*}
\end{proof}
\begin{remark}
    In the real category, since the modular class does not depend on 
    \(\rho\)-Berezin volume forms, 
    we also write \(\Modular_Q (M) \defeq \Modular_Q (M; \vol)\). 
\end{remark}
%
%

\begin{example}
    Consider the de Rham complex of a \(\rho\)-manifold \(M\) 
        with a coordinate \(x = (x^a)_a\). 
    By \thref{example_volume_of_shifted_tangent_bundle}, a \(\rho^\prime\)-manifold \(\Pi TM\) 
        has a \(\rho\)-Berezin volume form \(\vol\)
        which is locally \(D(x,dx) \cdot 1\). 
    It follows that 
    \begin{align*}
        &\phantom{{}=} \Diverg_{\vol} d \\
        &{}= \sum_{a = 1}^{n+m} \rho^\prime (|x^a|, |x^a| + |d|) \ppx{x^a}(dx^a)
        + \sum_{a = 1}^{n+m} \rho^\prime (|dx^a|, |dx^a| + |d|) \ppx{dx^a}(0) \\
        &{} = 0. 
    \end{align*}
    Therefore, \(\Modular_d (\Pi TM; \vol) = 0\). 
\end{example}
\begin{example}
    Consider the de Rham complex \((\Pi T \Complex^\times, Q = dz\ppx{z})\) 
        of \(\Complex^\times\). 
    Here we regard \(\Complex^\times\) as a non-graded holomorphic manifold. 
    We have a \(\rho\)-Berezin volume form \(\vol_1 \defeq D(z, dz) \cdot 1\). 
    We also have another volume form \(\vol_2 \defeq D(z, dz) \cdot z\). 
    We see that \(\Diverg_{\vol_1} Q = \ppx{z}(dz) = 0\) and 
    \(\Diverg_{\vol_2} Q = \frac{1}{z} \ppx{z} (dz \cdot z) = \frac{1}{z} dz\). 
    Since there is no holomorphic function \(f\) on \(\Complex^\times\) 
        which satisfies \(\ppx{z}(f) = \frac{1}{z}\), 
        we have \(\Modular_Q (M; \vol_2) \neq 0 = \Modular_Q (M; \vol_1)\). 
    Indeed, the two volume forms \(\vol_1\) and \(\vol_2\) are not equivalent. 
\end{example}
\begin{example}
    Let \((M, Q)\) be a \(\rho\)-Q-manifold 
        with a \(\rho\)-Berezin volume form \(\vol\). 
    Consider a \(\rho\)-manifold \([-i]T^\ast M\) \ \((i \in G)\) 
        with its homological vector field 
        \(\widetilde{Q} \defeq \dbrack{f_Q, -}\) 
        defined in \thref{example_Q_structure_on_i_shifted_cotangent_bundle}. 
    Denote by \(\widetilde{\vol}\) 
        the \(\rho\)-Berezin volume form on \([-i]T^\ast M\) induced by \(\vol\) 
        with a local description 
        \(\widetilde{\vol} = \widetilde{D}(x, x^\ast) \widetilde{s}(x, x^\ast)\). 
    Since \(\widetilde{s} = 1\) if \(i\) is even and 
        \(\widetilde{s} = s^2\) if \(i\) is odd, 
        its logarithm is written as 
    \begin{equation*}
        \log \widetilde{s} = (1 - \rho(i, i)) \log s. 
    \end{equation*}
    Hence, 
    \begin{equation*}
        \widetilde{Q}(\log \widetilde{s})
        = \widetilde{Q}((1 - \rho(i, i)) \log s)
        = (1 - \rho(i, i)) Q(\log s). 
    \end{equation*}
    The non-\(\widetilde{Q}\)-exact term of the divergence of \(\widetilde{Q}\) becomes
    \begin{align*}
        &{\phantom{{}=}} \sum_{a=1}^{n+m} \rho(|x^a|, |x^a| + |Q|) \ppx[Q^a]{x^a}
            - \sum_{a,b=1}^{n+m} \rho(|x_a^\ast|, |x_a^\ast| + |Q|) 
            \ppx{x_a^\ast}\left(\ppx[Q^b]{x^a}x_b^\ast\right) \\
        &{}= \sum_{a=1}^{n+m} \rho(|x^a|, |x^a| + |Q|) \ppx[Q^a]{x^a}
            - \sum_{a=1}^{n+m} \rho(|x_a^\ast|, |x_a^\ast| + |Q|) 
            \rho(-|x_a^\ast|,|Q|) \ppx[Q^a]{x^a} \\
        &{}= (1-\rho(i, i)) \sum_{a=1}^{n+m} \rho(|x^a|, |x^a|+|Q|) \ppx[Q^a]{x^a}. 
    \end{align*}
    Therefore, we obtain 
    \begin{equation*}
        \Diverg_{\widetilde{\vol}} \widetilde{Q} 
        = (1-\rho(i, i)) \Diverg_{\vol} Q
    \end{equation*}
    and
    \begin{equation*}
        \Modular_{\widetilde{Q}} ([-i]T^\ast M; \widetilde{\vol}) 
        = (1-\rho(i, i)) \Modular_Q (M, \vol). 
    \end{equation*}
\end{example}

\begin{example}
    We continue the example of noncommutative tori (\thref{example_Q_structure_on_noncommutative_tori}). 
    For the BRST differential \(Q\) on the functional algebra 
    \(A_\Theta\) on a noncommutative torus, 
    the divergence of \(Q\) with respect to 
    the trivial \(\rho\)-Berezin volume form \(\vol\) on a point \(M=\{\ast\}\) is 
    \begin{align*}
        \Diverg_{\vol} Q 
        &{}= -\sum_{a=1}^m \rho^\prime (|u^a|^\prime, |u^a|^\prime + |Q|^\prime) \ppx{u^a} (2 \pi \sqrt{-1} \eta^a u^a) \\
        &{}= -\sum_{a=1}^m \rho(|u^a|, |u^a|) 2 \pi \sqrt{-1} \eta^a \\
        &{}= -2 \pi \sqrt{-1}\sum_{a=1}^m \eta^a. 
    \end{align*}
    Hence, we obtain 
    \begin{equation*}
        \Modular_Q (\Pi\gee \times M; \vol) 
        = -2 \pi \sqrt{-1}\sum_{a=1}^m [\eta^a] \in H_Q^1(\Pi\gee \times M). 
    \end{equation*}
\end{example}